\documentclass[a4paper,12pt]{scrartcl}

\usepackage[utf8]{inputenc}
\usepackage[T1]{fontenc}
\usepackage[normalem]{ulem}
\usepackage{amsthm}
\usepackage{amssymb}
\usepackage{extarrows}
\usepackage[mathscr]{euscript}
\usepackage{hyperref}
\usepackage[shortcuts]{extdash}

\allowdisplaybreaks

\newcommand\restrict[1]{\raisebox{-.3ex}{$|$}_{#1}}
\newcommand{\kk}{\Bbbk} 
\newcommand{\ZZ}{\mathbb{Z}}
\newcommand{\NN}{\mathbb{N}}
\newcommand{\II}{\mathbb{I}}
\newcommand{\eqbydef}{\mathrel{\overset{\makebox[0pt]{\mbox{\normalfont\tiny def}}}{=}}}
\newcommand{\lto}{\longrightarrow}
\newcommand{\lmapsto}{\longmapsto}
\newcommand{\eps}{\varepsilon}
\newcommand{\linhuell}{\operatorname{span}}
\newcommand{\id}[1]{\operatorname{id}_{#1}}
\newcommand{\Hom}[2]{ \operatorname{Hom}_{#1}\! \left( #2 \right) }
\newcommand{\Ext}{ \operatorname{Ext} }
\newcommand{\End}[2]{\operatorname{End}_{#1}(#2)}
\newcommand{\Mat}[2]{\operatorname{Mat}_{#1}(#2)}
\newcommand{\Aut}[2]{ \operatorname{Aut}_{#1}\! \left( #2 \right) }
\newcommand{\ot}{\otimes}
\newcommand{\Cent}{\operatorname{Cent}}
\newcommand{\co}{\operatorname{co}}

\newcommand{\tild}[1]{{\widetilde{#1}}}
\newcommand{\lmod}[1]{{#1}\!\operatorname{--mod}}

\newcommand{\soc}{{\operatorname{soc}}}
\newcommand{\sgn}{{\operatorname{sgn}}}
\newcommand{\opp}{{\operatorname{opp}}}

\theoremstyle{plain}

\newcounter{dummy}

\newtheorem{theorem}[dummy]{Theorem}
\newtheorem{proposition}[dummy]{Proposition}
\newtheorem{lemma}[dummy]{Lemma}
\newtheorem{corollary}[dummy]{Corollary}
\newtheorem{conjecture}[dummy]{Conjecture}

\theoremstyle{definition}

\newtheorem{definition}[dummy]{Definition}
\newtheorem{example}[dummy]{Example}

\theoremstyle{remark}

\newtheorem{remark}[dummy]{Remark}

\usepackage[top=2.0cm, bottom=2.0cm, left=2.0cm, right=2.0cm]{geometry}
\addtokomafont{sectioning}{\rmfamily}
\newcommand{\ph}{p}
\newcommand{\phs}[1]{{p_{#1}}}

\title{%
\begin{flushright}
	\vspace{-1.8cm} \normalfont{\small{\textsf{[ZMP-HH/19-20]}}}\\
	\vspace{-0.5cm} \normalfont{\small{\textsf{Hamburger Beiträge zur Mathematik Nr. 811}}}\\
	\vspace{-0.5cm} \normalfont{\small{\textsf{Oktober 2019}}}
\end{flushright}
\vspace{0.5cm}
On isotypic decompositions for non-semisimple Hopf algebras
}

\author{%
	\!\!\!\!\!\!\!Vincent Koppen$^*$ \quad
	Ehud Meir$^\#$ \quad
	Christoph Schweigert$^*$
	\\[0.25cm]
	\hspace{-1.8cm}  \normalsize{\texttt{\href{mailto:vincent.koppen@uni-hamburg.de}{vincent.koppen@uni-hamburg.de}}} \\  %
	\hspace{-1.8cm}  \normalsize{\texttt{\href{mailto:ehud.meir@abdn.ac.uk}{ehud.meir@abdn.ac.uk}}}\\
	\hspace{-1.8cm}  \normalsize{\texttt{\href{mailto:christoph.schweigert@uni-hamburg.de}{christoph.schweigert@uni-hamburg.de}}} \\[0.1cm]
	\hspace{-1.2cm} {\normalsize\slshape $^*$Fachbereich Mathematik, Universit\"{a}t Hamburg, Germany}\\[-0.1cm]
	\hspace{-1.2cm} {\normalsize\slshape $^\#$Institute of Mathematics, University of Aberdeen, United Kingdom
	}\\[-0.75cm]
}

\date{}

\begin{document}

\maketitle

\begin{abstract}
In this paper we study the isotypic decomposition of the regular module of a finite-dimensional Hopf algebra over an algebraically closed field of characteristic zero.
For a semisimple Hopf algebra, the idempotents realizing the isotypic decomposition can be explicitly expressed in terms of characters and the Haar integral.
In this paper we investigate Hopf algebras with the Chevalley property, which are not necessarily semisimple.
We find explicit expressions for idempotents in terms of Hopf-algebraic data, where the Haar integral is replaced by the regular character of the dual Hopf algebra.
For a large class of Hopf algebras, these are shown to form a complete set of orthogonal idempotents.
We give an example which illustrates that the Chevalley property is crucial.
\end{abstract}


\section{Introduction}
In this paper we study the decomposition of the regular module of a finite-dimensional Hopf algebra into isotypic components.
More precisely, let $\kk$ be an algebraically closed field of characteristic zero and $H$ a finite-dimensional Hopf algebra over $\kk$. 

If $H$ is semisimple, the Artin-Wedderburn theorem implies that as a left $H$-module $H$ decomposes into the direct sum of submodules $H_i$ isomorphic to the $\dim(S_i)$-fold direct sum $S_i^{\oplus \dim S_i}$ of the simple $H$-module $S_i$.
Here $i$ runs over the set $I$ of isomorphism classes of simple $H$-modules.
The decomposition $H=\bigoplus_{i\in I} H_i$ is called the \emph{isotypic decomposition} of $H$, seen as a left $H$-module, into its \emph{isotypic components} $H_i$.
It can be also described by the central orthogonal idempotents $(e_i)_{i\in I}$ in $H$ such that $e_i \in H_i$ and $\sum_{i\in I} e_i = 1$.
Then $H_i = He_i$ and the projection from $H = \bigoplus_{i \in I} H_i$ onto the direct summand $H_j$ is given by right multiplication by $e_j$ for all $j \in I$.

So far this only uses the algebra structure of $H$.
The following idea is well-known and lies at the heart of the theory of representations of a finite group.
For a semisimple Hopf algebra $H$ over $\kk$ with antipode $S$, the central orthogonal idempotents $e_i$ can be described explicitly in terms of the Haar integral and the irreducible characters of $H$ by the following \emph{character-projector formula} \cite[Cor.~4.6]{schneider} 
\begin{equation} \label{eq:character-projector-formula} 
e_i = \dim(S_i) \chi_i(S(\ell_{(1)})) \ell_{(2)}.
\end{equation}

Here, Sweedler notation is understood, and $\ell \in H$ is the \emph{Haar integral} for $H$, the unique (two-sided) integral of $H$, normalised to $\eps(\ell)=1$, which exists due to the Maschke theorem for semisimple Hopf algebras \cite[Theorem~5.1.8]{sweedler}.
The functional $\chi_i : H \lto \kk$ here is the character of the simple $H$-module $S_i$.

In this paper, we study generalizations of the character-projector formula \eqref{eq:character-projector-formula} for finite-dimensional Hopf algebras that are not necessarily semisimple.
Hence, we do not have a Haar integral at our disposal.
Instead, we use  the character of the regular representation of the Hopf algebra $H^*$ dual to $H$.
While for semisimple algebras there is a unique isotypic decomposition, in the non-semisimple case such decompositions are in general not unique anymore.
Our aim in this paper is to nevertheless construct one {\em explicit} decomposition using the Hopf-algebraic structure.

We obtain the strongest results for Hopf algebras with the Chevalley property, see Definition \ref{def:chevalley}.
This is a large class of finite-dimensional Hopf algebras, including semisimple Hopf algebras and basic Hopf algebras, i.e.\! Hopf algebras for which all simple modules are one-dimensional, the Hopf algebras dual to pointed Hopf algebras. 

Our main results are as follows:
for a finite-dimensional Hopf algebra with the Chevalley property, we give in Theorem \ref{thm:idempotents_for_one-dimensional_simples} an explicit idempotent for each one-dimensional simple module.
In Theorem \ref{thm:the-phs-sum-up-to-one}, we exhibit a necessary and sufficient condition involving the so-called Hecke algebra of the trivial representation (see Definition \ref{def:hecke-algebra}) ensuring that these idempotents form a complete set in the sense that they sum up to the identity.

In Conjecture \ref{conjecture}, we propose an explicit generalization of the character-projector formula \eqref{eq:character-projector-formula} for finite-dimensional Hopf algebras with the Chevalley property.
(The Chevalley property is essential, as witnessed by the counterexample given in Example \ref{exa:non-chevalley}.)
The two main theorems \ref{thm:idempotents_for_one-dimensional_simples} and \ref{thm:the-phs-sum-up-to-one} imply our Conjecture \ref{conjecture} for basic Hopf algebras that satisfy the condition on the Hecke algebra, as summarized in Corollary \ref{corollary}. 
Furthermore, in Proposition \ref{prop:chevalley-and-cochevalley} we prove that Conjecture \ref{conjecture} holds for Hopf algebras which have the Chevalley property and the dual Chevalley property.

Our original motivation is in the context of Kitaev models.
Here, a vector space is assigned to a surface with a triangulation.
The Kitaev construction assigns to each edge a copy of a Hopf algebra $H$ and to the surface with triangulation the tensor product of these Hopf algebras.
One singles out a subspace of this tensor product by specifying commuting projectors for each 0-cell and each 2-cell.
In a generalization, one assigns to a pair, consisting of a 0-cell and an adjacent 2-cell, called a site, a representation of the Drinfeld double $D(H)$ and then singles out different subspaces by using the projectors to isotypic components of $D(H)$, see \cite{balKi} for a related construction.
This shows that it is interesting also for applications to have such projectors explicitly available.
\ \\

This paper is organised as follows: 
In Section \ref{sec:finite-dimensional} we first review the definition of isotypic decompositions for finite-dimensional algebras and then obtain some first preliminary results about them.
In particular, in Proposition \ref{prop:characterization-of-semisimplicity} we give a characterisation of the semisimplicity of a Hopf algebra in terms of the centrality of the idempotent associated to the trivial isotypic component.
Section \ref{sec:chevalley} contains our main results for general finite-dimensional Hopf algebras with the Chevalley property.
Finally, in Section \ref{sec:examples} we first illustrate our results with an example of a basic Hopf algebra (Subsection \ref{subsec:example-basic-hopf-algebra}) and then provide further evidence for Conjecture \ref{conjecture} by studying in Subsection \ref{subsec:example-72-dim} an example of a Hopf algebra with the Chevalley property that is not covered by our general results in Section \ref{sec:chevalley}.

\begin{section}*{Acknowledgments}
We would like to thank Raz Vakil for his patient guidance in Magma, Johannes Berger for discussions and Matthieu Faitg for helpful comments.
VK, EM and CS are partially supported by the RTG 1670 “Mathematics inspired by String theory and Quantum Field Theory”.
CS is also partially supported by the Deutsche Forschungsgemeinschaft (DFG, German Research Foundation) under Germany’s Excellence Strategy - EXC 2121 “Quantum Universe”- QT.2.
\end{section}


\section{Isotypic decompositions for finite-dimensional algebras} \label{sec:finite-dimensional}

Let $H$ be a (not necessarily semisimple) finite-dimensional algebra over $\kk$.
Let, as before, $I$ denote the (finite) set of isomorphism classes of simple $H$-modules.
Then, as a projective left $H$-module, $H$ possesses a direct sum decomposition into projective $H$-submodules $H_i$,
\begin{equation} \label{eq:isotypic-decomposition}
H = \bigoplus_{i \in I} H_i,
\end{equation}
where $H_i \cong P_i^{\oplus n_i}$ is a direct sum of projective indecomposable submodules of the same isomorphism type $P_i$, the projective cover of the simple $H$-module given by $i \in I$.
\begin{definition}
We call $H_i$ an \emph{$i$-isotypic component} of $H$, for $i \in I$, and a direct sum decomposition into isotypic components an \emph{isotypic decomposition} of $H$.
\end{definition}
Specifying an isotypic decomposition is equivalent to specifying the corresponding orthogonal idempotents $(p_i)_{i\in I}$ such that $p_i \in H_i$ and $\sum_{i \in I} p_i = 1$.

\begin{remark}
Isotypic decompositions can clearly be defined for any projective left module over $H$.
However, in general there does not exist a description in terms of orthogonal idempotents in $H$, since for this we use that left $H$-module endomorphisms of $H$ are in bijection with right multiplications with elements of $H$: $\End{H}{H} \cong H^\opp$.
\end{remark}

By the Krull-Schmidt theorem, the multiplicities $n_i$ of the indecomposable modules inside each $H_i$ are unique for any isotypic decomposition.
In fact, they are given by the dimensions of the simple $H$-modules.
Indeed, let $S_i$ be a simple $H$-module in the isomorphism class $i \in I$ and let $H_i \cong P_i^{\oplus n_i}$, where $P_i$ is the projective cover of $S_i$ and $n_i \in \NN$.
Then we have an isomorphism of vector spaces
\begin{equation*}
S_i \cong \Hom{H}{H, S_i}
\cong \Hom{H}{\oplus_{j \in I} P_j^{\oplus n_j}, S_i}
\cong \Hom{H}{P_i,S_i}^{\oplus n_i}.
\end{equation*}
Since $\Hom{H}{P_i, S_i}$ is one-dimensional, this implies that $n_i = \dim(S_i)$.

Another point of view on isotypic decompositions is the following.
Let $J \subseteq H$ be the Jacobson radical of the finite-dimensional algebra $H$, i.e. the maximal nilpotent ideal of $H$ (as a general reference see \cite{bresar}).
Then the algebra $H/J$ is the maximal semisimple quotient algebra of $H$ with natural surjection of algebras $\pi : H \lto H/J$.

\begin{lemma} \label{lem:isotypic-decomposition-as-a-lift-of-the-center}
An isotypic decomposition of $H$ is equivalent to an algebra map $s : Z(H/J) \lto H$ such that $\pi \circ s = \id{Z(H/J)}$.
\end{lemma}
\begin{proof}
Indeed, given an isotypic decomposition $H = \bigoplus_{i \in I} H p_i$, mapping $Z(H/J) \ni e_i \mapsto p_i \in H$, where $e_i$ are the central orthogonal idempotents of the semisimple algebra $H/J$, gives us such an algebra map $s : Z(H/J) \lto H$ because $Z(H/J) = \text{span}_{\kk}\{e_i\}$.

Conversely, given an algebra map $s : Z(H/J) \lto H$ such that $\pi \circ s = \id{Z(H/J)}$, the images of the central orthogonal idempotents $e_i \in H/J$ give us the orthogonal idempotents $p_i := s(e_i)$ of an isotypic decomposition of $H$.
For this we need to show that $H p_i$ is a projective cover of $S_i^{\oplus\dim(S_i)}$.
Indeed, $H p_i$ is a projective cover of $H p_i / J(H p_i)$ and we have an $H$-module isomorphism $H p_i / J(H p_i) \cong (H/J) e_i \cong S_i^{\oplus \dim S_i}$.
The first isomorphism follows from $J(H p_i) = H p_i \cap J(H) = \ker(\pi\restrict{H p_i})$, together with the fact that $(H/J) e_i$ is the image of the restricted quotient map $\pi\restrict{H p_i} : H p_i \lto H/J$.
\end{proof}

\subsection{(Non-)uniqueness of isotypic decompositions}

In general, an isotypic decomposition is not unique.
We can characterize the uniqueness of such a decomposition as follows.
\begin{lemma} \label{lem:uniqueness-of-isotypic-decomposition}
Let $H$ be a finite-dimensional algebra over $\kk$.
A direct sum decomposition $H = \bigoplus_{i \in J} H_i$ into left $H$-submodules, where the isomorphism types of the summands $H_i$ are prescribed, is unique if and only if any left $H$-module automorphism of $H$ commutes with the projections of the direct sum $\bigoplus_{i \in J} H_i$.
\end{lemma}
\begin{proof}
Let $H = \bigoplus_{i \in J} H_i$ be a unique decomposition into components of prescribed isomorphism type and let $\varphi : H\lto H$ be an $H$-module automorphism.
Then $H = \bigoplus_i \varphi(H_i)$ together with the projections $(\varphi \circ p_i \circ \varphi^{-1} : H \to \varphi(H_i))_i$ is also such a decomposition and therefore $H_i = \varphi(H_i) \in H$ and $p_i = \varphi \circ p_i \circ \varphi^{-1}$ for all $i \in J$.

It remains to prove the implication in the other direction.
For this assume that any $H$-automorphism $\varphi : H\lto H$ commutes with the projections $p_i : H \lto  H_i$ for all $i \in J$.
Now let $H = \bigoplus_i H'_i$ be another decomposition into components of the prescribed isomorphism type with projections $p'_i : H \lto H'_i$.
There are isomorphisms $\varphi_i : H_i \lto H'_i$.
Together they give an isomorphism $\varphi = \bigoplus_i \varphi_i : H = \bigoplus_i H_i \lto \bigoplus_i H'_i = H$, which by construction satisfies $\varphi \circ p_i = p'_i \circ \varphi$ for all $i \in J$.
But by assumption an $H$-automorphism $\varphi$ commutes with the projections $p_i : H \lto H_i$ for all $i \in J$, i.e. we have $\varphi \circ p_i = p_i \circ \varphi$ for all $i \in J$.
Together this implies $p'_i \circ \varphi = p_i \circ \varphi$ for all $i \in J$ and by invertibility of $\varphi$ this proves the claim that $p_i = p'_i$.
\end{proof}
\begin{remark}
Furthermore, we can describe the set of decompositions of $H$ into isotypic components as follows.
Choose one such decomposition $(p_i : H \lto H_i)_i$.
Mapping an $H$-linear automorphism $\varphi \in \Aut{H}{H}$ to the decomposition $(\varphi \circ p_i \circ \varphi^{-1} : H \lto \varphi(H_i))_i$ induces a bijection
\[ \text{Aut}_H (H)\ /\ \prod_i \text{Aut}_H(H_i) \xlongrightarrow{\sim} \Big\{ (p'_i : H \to H'_i)_i \text{ isotypic decomposition}  \Big\} . \]
Denoting by $\Cent_{H^\times}\{ p_i \vert i \in I \}$ the centralizer of the set $(p_i)_{i\in I}$ in $H^\times$, we can thus also describe the set of isotypic decompositions as the homogeneous set
\[ H^\times / \Cent_{H^\times}\{ p_i \vert i \in I \} . \]
\end{remark}

If $H$ is a semisimple Hopf algebra, then the idempotents $e_i$ in equation \eqref{eq:character-projector-formula} giving us the isotypic decomposition are central (as they are for any semisimple algebra by the Artin-Wedderburn theorem), implying by Lemma \ref{lem:uniqueness-of-isotypic-decomposition} the uniqueness of the isotypic decomposition in the semisimple case.

Conversely, we obtain the following characterization of semisimplicity for a Hopf algebra $H$:
\begin{proposition} \label{prop:characterization-of-semisimplicity}
A finite-dimensional Hopf algebra $H$ over $\kk$ is semisimple if and only if there exists a decomposition $H = \bigoplus_{i\in I} H e_i$ into isotypic components such that $e_{\II} \in H$, the idempotent corresponding to the trivial $H$-module, is central.
\end{proposition}
\begin{proof}
The only-if part of the statement is implied by the Artin-Wedderburn theorem.

For the rest of the proof assume that $e_\II$ is central.
This implies that $\Ext_H^1(\II,S)=0$ for any non-trivial simple $H$-module $S$ as we show next.
Let $0 \to S \to M \to \II \to 0$ be a short exact sequence in $\lmod{H}$, where $M$ is an arbitrary $H$-module.
Since $e_\II\in H$ is central, acting with this element defines an $H$-module morphism on any $H$-module, in particular $e_\II : S \to S$.
If there were an $x\in S$ such that $e_\II.x \neq 0$, then this would define a non-zero $H$-module map $H e_\II \lto S, h e_\II \lmapsto h e_\II . x$.
But for a simple $H$-module $S$ non-isomorphic to the trivial one $\II$, this does not exist, since $H e_\II$ is the projective cover of $\II$.
Hence, we must have $e_\II . S = 0$.
This implies that we obtain a well-defined morphism $e_\II : M / S \to M$, which provides a splitting of the short exact sequence, since $M/S \cong \II$ by exactness of the sequence and, hence, $e_\II$ acts on $M/S$ as the identity.
We have thus shown that $\Ext_H^1(\II,S)=0$ for any non-trivial simple $H$-module $S$, using that $e_\II$ is central.

Due to Theorem 4.4.1 in \cite{etingof+et_al}, we also have $\Ext_H^1(\II,\II)=0$.

Now let $N$ be an arbitrary $H$-module.
Since a short exact sequence of $H$-modules induces a long exact sequence of corresponding Ext groups, we can use a composition series for $N$ to inductively show that there exists a simple $H$-module $S$ (the smallest module in the composition series) such that $\Ext_H^1(\II,S)$ surjects to $\Ext_H^1(\II,N)$.
Since we have shown that $\Ext_H^1(\II,S)=0$, this implies that $\Ext_H^1(\II,N)=0$ for all $N$, and hence the trivial $H$-module $\II$ is projective.
This implies that every $H$-module is projective, since the tensor product of a projective module with any other module is projective.
We conclude that $H$ is semisimple.
\end{proof}

\begin{example} \label{exa:sweedlers-algebra}
Let $H=H_4=\kk\langle g,x\rangle / (g^2=1, x^2=0, gx=-xg)$ be Sweedler's four-dimensional Hopf algebra, which reappears in more detail in Example \ref{exa:sweedlers-hopf-algebra}.
Consider the decomposition $H = P_0 \oplus P_1 := H \frac{1+g}{2} \oplus H \frac{1-g}{2}$.
Then there exists an automorphism $\varphi : H \to H$ of $H$ as a left $H$-module such that $\varphi(P_0)\neq P_0$.
Indeed, let $\varphi$ be given by right multiplication by the invertible element $1+\frac{1+g}{2}x$.
This does not commute with the element $\frac{1+g}{2}$, as can be easily computed.
Hence $\frac{1+g}{2}$ is not central and $H_4$ is not semisimple.

We can therefore see that $$ H = H \frac{1+g}{2} \oplus H \frac{1-g}{2} $$ and $$ H = H \frac{1+g}{2}(1+x) \oplus H \frac{1-g}{2} $$ are two different isotypic decompositions for $H_4$.
\end{example}

\subsection{Isotypic decompositions for self-injective algebras}

When we are dealing with a self-injective algebra $H$, i.e.\ the regular $H$-module is injective, then finding an isotypic decomposition simplifies to finding isotypic components individually, as we will show in this subsection.
In particular, this applies to Hopf algebras since they are Frobenius algebras and hence self-injective.

For a simple $H$-module $S_i$, $i \in I$, write $\phi(i)$ for the (isomorphism class of) the socle $\soc(P_i)$ of the projective cover $P_i$ of $S_i$.
Since $P_i$ is also injective, because $H$ is a self-injective algebra, it holds that $P_i$ is the injective envelope of $\phi(i)$.
Since $P_i$ is indecomposable, this means that $\phi(i)$ is simple.
So this gives us a bijection $\phi: I \lto I$, which is called the \emph{Nakayama permutation} of $H$(cf. \cite{farnsteiner}). 
In particular, if $i \neq j$, we have $\soc(P_i) \neq \soc(P_j)$.

\begin{proposition} \label{prop:isotypic_decomposition}
Assume that for any $i\in I$, $H_i \subseteq H$ is an $i$-isotypic component of $H$, i.e. $H_i \cong P_i^{\oplus \dim(S_i)}$.
Then the map \[ \Psi:\bigoplus_{i\in I}H_i\lto H \] is an isomorphism of left $H$-modules.
\end{proposition}
\begin{proof}
Due to dimension considerations it is enough to prove that $\Psi$ is injective.
Assume the opposite.
Then choose a minimal subset $J \subseteq I$ such that $\Psi\vert : \bigoplus_{i\in J} H_i\lto H$ is not injective.
Then for some $i \in J$ we have that $\Psi\vert : \bigoplus_{j \in J \setminus\{i\}} H_j \lto H$ is injective and $H_i\cap \Psi(\bigoplus_{j \in J \setminus\{i\}} H_j) \neq 0$.
If two modules intersect, then there is a simple module contained in the intersection.
But the biggest semisimple submodule of $H_i$ is the socle $\soc(H_i) \cong \soc(P_i)^{\oplus \dim(S_i)} = \phi(i)^{\oplus \dim(S_i)}$.
The biggest semisimple submodule of the other module is
\[ \bigoplus_{j \in J \setminus\{i\}}\soc(H_j)\cong \bigoplus_{j \in J \setminus\{i\}} \soc(P_j)^{\oplus \dim(S_j)} = \bigoplus_{j \in J \setminus\{i\}} \phi(j)^{\oplus \dim(S_j)} . \]
Since these two modules do not have any common submodule (up to isomorphism), we arrive at a contradiction.
\end{proof}

So the conclusion of this proposition is:
Once we pick for every $i \in I$ a submodule $H_i$ which is of the right isomorphism type (that is, it is an $i$-isotypic component), this will give us a direct sum decomposition of $H$ into isotypic components.

\section{Isotypic decompositions for Hopf algebras with the Chevalley property} \label{sec:chevalley}

As we have seen, isotypic decompositions of non-semisimple Hopf algebras are in general not unique.
Our goal is to nevertheless construct one explicit isotypic decomposition for a Hopf algebra by generalizing to the non-semisimple case the idempotents given by the character-projector formula \eqref{eq:character-projector-formula}, which makes use of the additional Hopf-algebraic structure such as the Haar integral.

Note that the idempotent $e_{\II}$ for the isotypic component corresponding to the trivial module is given by the Haar integral: $e_{\II} = (\eps(S(\ell_{(1)}) \ell_{(2)})(\Delta(\ell)) = \ell$.
In a non-semisimple Hopf algebra any (say, left) integral $\ell \in H$ satisfies $\eps(\ell)=0$ and, hence, $\ell^2 = \eps(\ell) \ell = 0$.
So $\ell$ is not an idempotent anymore.
Therefore the character-projector formula \eqref{eq:character-projector-formula} does not generalize to anything desirable in terms of idempotents for an isotypic decomposition in the non-semisimple case.

Instead we want to take into account that for a semisimple Hopf algebra the Haar integral coincides with the (appropriately normalized) character of the regular representation of the dual algebra:
\begin{proposition}[Prop.~1~b) in \cite{larsonRadford}] \label{prop:haar-integral-given-by-character}
Let $H$ be a semisimple finite-dimensional Hopf algebra over $\kk$.
Then the Haar integral of $H$ is equal to $\ph := \frac{1}{\dim(H)} \chi_{H^*}$, where $\chi_{H^*} \in H^{**} \cong H$ is the character of the regular $H^*$-module.
\end{proposition}
This proposition motivates us to consider, for our purposes, the character of the regular $H^*$-module as the appropriate generalization of the Haar integral to the non-semisimple case:
\begin{equation}
\ph := \frac{1}{\dim(H)} \chi_{H^*} \in H^{**} \cong H.
\end{equation}
Indeed,
\begin{itemize}
\item
We still have that $\ph$ is an idempotent:
\begin{align*}
\ph^2 &= \frac{1}{\dim(H)^2}\chi_{H^*}\chi_{H^*}
= \frac{1}{\dim(H)^2}\chi_{H^* \otimes H^*}
= \frac{1}{\dim(H)^2}\chi_{H^* \otimes H^*_\text{triv}} \\
&= \frac{1}{\dim(H)^2}\chi_{H^*}\chi_{H^*_\text{triv}}
= \frac{1}{\dim(H)^2}\chi_{H^*}\dim(H)
= \frac{1}{\dim(H)}\chi_{H^*}
= \ph,
\end{align*}
where we have used the isomorphism of $H^*$-modules $H^* \otimes H^* \to H^* \otimes H^*_\text{triv}, f \otimes g \mapsto f_{(1)} \otimes S(f_{(2)}).g$.
\item
Another basic property of $\ph$ is that $\eps(\ph) = 1$, since $\eps(\chi_{H^*}) = \chi_{H^*}(\eps) = \dim(H)$.
\item
Moreover, the cyclicity of the trace implies that $\ph \in H$ is cocommutative.
\item
A difference from the semisimple case is that $\ph$ is in general not central when $H$ is not semisimple.
But, due to Proposition \ref{prop:characterization-of-semisimplicity}, this is exactly what we expect.
\end{itemize}

\begin{example} \label{exa:sweedlers-hopf-algebra}
Let us consider as a non-semisimple example where the element $\ph$ is not equal to any integral, the four-dimensional Sweedler Hopf algebra (cf. Example \ref{exa:sweedlers-algebra})
\[ H_4 = \kk\langle g,x\rangle / (g^2=1, x^2=0, gx=-xg) \]
The co-multiplication is given by $\Delta (g) = g \ot g$ and $\Delta (x) = x \ot 1 + g \ot x$.

As can be straightforwardly verified, the space of left integrals of $H_4$ is $I_\ell = \kk (1 + g)x$ and the space of right integrals is $\kk x(1 + g)$.
However, $\ph \in H_4$ is equal to $\frac{1}{2}(1 + g)$.

For later reference, let us list the simple modules of Sweedler's Hopf algebra $H_4$.
The Jacobson radical $J$ of $H_4$ is $\kk \{ x, gx \} \subseteq H_4$ and the maximal semisimple quotient $H_4 / J$ of $H_4$ can then be identified with the group algebra $\kk \{1,g\}$ of the cyclic group of order $2$.
Hence, there are two non-isomorphic simple $H_4$-modules, both one-dimensional: the trivial one, denoted $\kk_+$, sending $g \mapsto 1$ and $x \mapsto 0$ and the other one, denoted $\kk_-$, mapping $g \mapsto -1$ and $x \mapsto 0$.

Let us see what the character-projector formula \eqref{eq:character-projector-formula} would give us in this example if we replace the Haar integral $\ell$ in the formula by the idempotent $\ph$.
For the trivial module, we of course get $\ph = \frac{1}{2}(1 + g)$ itself and for the non-trivial simple $H_4$-module we obtain $\frac{1}{2}(1-g)$.
These are indeed orthogonal and both idempotent.
\end{example}

The Hopf algebra $H_4$ in the previous example has only one-dimensional simple modules, so in particular satisfies the so-called Chevalley property.
For the remainder of this paper, we will restrict our attention to the better behaved subclass of Hopf algebras with the Chevalley property.

\begin{definition} \label{def:chevalley}
Let $H$ be a Hopf algebra.
$H$ has the \emph{Chevalley property} if the tensor product of any two semisimple $H$-modules is again a semisimple $H$-module.
\end{definition}
\begin{remark} \label{def:basic}
A Hopf algebra $H$, all of whose simple modules are one-dimensional, clearly possesses the Chevalley property.
We call such an algebra \emph{basic}.
\end{remark}

Another characterization of the Chevalley property is the following (cf. \cite[Proposition~4.2]{andEtGel}):
\begin{lemma}
$H$ has the Chevalley property if and only if its Jacobson radical $J$ is a Hopf ideal in $H$, that is $J$ is an ideal such that $\Delta (J) \subseteq J \otimes H + H \otimes J$, $\eps(J)=0$ and $S(J) \subseteq J$, or in other words, $H/J$ has the structure of a Hopf algebra such that the quotient map $\pi : H \lto H/J$ is a Hopf algebra morphism.
\end{lemma}
The class of Hopf algebras with the Chevalley property in particular includes the semisimple Hopf algebras, but is much larger than that.
In fact, many known examples of finite-dimensional Hopf algebras over $\kk$ have either the Chevalley property or a dual Hopf algebra with the Chevalley property.
The latter case includes in particular the pointed Hopf algebras that have been much studied by N. Andruskiewitsch, H.-J. Schneider et al.
For example the pointed Hopf algebras with abelian group of group-like elements were classified in \cite{AS10} (under a mild restriction on the order of the group).
Hopf algebras with the Chevalley property were also studied in \cite{andEtGel, andGM, andVay}.

In view of Proposition \ref{prop:haar-integral-given-by-character} and the observations following it, we conjecture the following character-projector formula for non-semisimple Hopf algebras:
\begin{conjecture} \label{conjecture}
Let $H$ be a finite-dimensional Hopf algebra over $\kk$ with the Chevalley property (cf. Def. \ref{def:chevalley}).
Then we conjecture that the elements
\begin{equation} \label{eq:character-projector-formula-nonssi}
\phs{i} := \dim(S_i) \chi_i(S(\ph_{(1)})) \ph_{(2)} \in H,
\end{equation}
for $i \in I$, where $I$ is the set of isomorphism classes of simple $H$-modules, define a set of orthogonal idempotents of $H$ such that $H = \bigoplus_{i\in I} H \phs{i}$ is an isotypic decomposition.
\end{conjecture}
\begin{remark}
When $H$ has the Chevalley property, the quotient map $\pi : H\lto H/J$ sends $\{ \phs{i} \}_{i \in I}$ to the canonical basis of central orthogonal idempotents of the center of the semisimple algebra $H/J$.
The restriction of the quotient map to the subalgebra $\kk\langle \phs{i}\rangle \subseteq H$ generated by $\{ \phs{i} \}_{i \in I}$ then gives us a surjection $$\pi : \kk\langle \phs{i}\rangle \lto Z(H/J).$$ onto a semisimple algebra.
It is known that any such surjection splits as an algebra map, and if $s:Z(H/J)\to \kk\langle \phs{i}\rangle$ is a section, then $s\pi(\phs{i})$ will give us a basis of orthogonal idempotents which are moreover non-commutative polynomials in the elements $\phs{i}$.
However, our conjecture asks if the character-projector formula generalises directly, without any adjustments.
\end{remark}

In order to justify our assumption of the Chevalley property for this conjecture, we will give a counter-example: 

\begin{example} \label{exa:non-chevalley}
Let us give a counter-example to the above conjecture for a Hopf algebra that does not possess the Chevalley property.

Let $\mu \in \kk$.
Then we define an $8$-dimensional Hopf algebra $H(\mu)$ over $\kk$, a deformation of the so-called \emph{double cover} of Sweedler's $4$-dimensional Hopf algebra, as follows.
As an algebra it is generated by elements $g$ and $x$ with the relations
\begin{align*}
	g^4 &= 1 \\
	g x g^{-1} &= -x \\
	x^2 &= \frac{\mu}{2}  (1 - g^2) ,
\end{align*}
and the co-multiplication is given by $\Delta (g) = g \ot g$ and $\Delta (x) = x \ot 1 + g \ot x$.
From this one can compute that $\ph = \frac{1}{\dim(H_\mu)} \chi_{H(\mu)^*} = \frac{1}{4}(1+g+g^2+g^3)$.
Let us compute the remaining $\phs{i}$ from the character-projector formula \eqref{eq:character-projector-formula-nonssi}:

To this end we have to determine the simple $H(\mu)$-modules and compute their characters.
Since $\frac{1+g^2}{2}$ and $\frac{1-g^2}{2}$ are central orthogonal idempotents, $H(\mu)$ decomposes as a direct sum of algebras
\[ H(\mu) = H(\mu) \frac{1+g^2}{2} \oplus H(\mu) \frac{1-g^2}{2} =: H(\mu)^+ \oplus H(\mu)^- . \]
Denoting $\gamma := g (\frac{1+g^2}{2})$ and $\xi := x (\frac{1+g^2}{2})$, as an algebra $H(\mu)^+$ has the unit $1_+ := \frac{1+g^2}{2}$ and is generated by $\gamma$ and $\xi$ satisfying the relations
\begin{align*}
	\gamma^2 &= 1_+, \\
	\gamma \xi \gamma^{-1} &= -\xi, \\
	\xi^2 &= 0 ,
\end{align*}
and hence is isomorphic to the four-dimensional Sweedler algebra $H_4$.

On the other hand, denoting $G := g (\frac{1-g^2}{2})$ and $X := x (\frac{1-g^2}{2})$, as an algebra $H(\mu)^-$ has the unit $1_- := \frac{1-g^2}{2}$ and is generated by $G$ and $X$ satisfying the relations
\begin{align*}
	G^2 &= -1_-, \\
	G X G^{-1} &= -X, \\
	X^2 &= \mu 1_-
\end{align*}
Now we have to distinguish the two cases $\mu = 0$ and $\mu \neq 0$.

Let us first assume that $\mu = 0$.
Then setting $\tild{G} := i G$, the elements $\tild{G}$ and $X$ satisfy the relations of the four-dimensional Sweedler algebra $H_4$, so that we have that, as an algebra $H(0)$, is isomorphic to the direct sum of two Sweedler algebras: $H(0) \cong H_4 \oplus H_4$.
We have seen in Example \ref{exa:sweedlers-hopf-algebra} that Sweedler's Hopf algebra has two one-dimensional simple modules $\kk_+$ and $\kk_-$, so that $H(0)$ therefore has four one-dimensional simple modules: $\kk^+_+$ and $\kk^+_-$ for the $H(0)^+$-part and $\kk^-_+$ and $\kk^-_-$ for the $H(0)^-$-part.
The corresponding orthogonal idempotents $\phs{i}$ according to the character-projector formula \eqref{eq:character-projector-formula-nonssi} are $\ph = \phs{\kk^+_+} = \frac{1}{4}(1+g+g^2+g^3)$ and $\phs{\kk^+_-} = \frac{1}{4}(1-g+g^2-g^3)$ for the two simple modules of $H(0)^+ \cong H_4$, and $\phs{\kk^-_+} = \frac{1}{4}(1+ig-g^2-ig^3)$ and $\phs{\kk^-_-} = \frac{1}{4}(1-ig-g^2+ig^3)$ for the two simple modules of $H(0)^- \cong H_4$.

Now consider the case $\mu \neq 0$.
Then it turns out that the algebra $H(\mu)^-$ is isomorphic to the matrix algebra $\Mat{2}{\kk}$, by identifying $G$ with $\begin{pmatrix} i & 0 \\ 0 & -i \end{pmatrix}$ and $X$ with $\begin{pmatrix} 0 & \mu \\ 1 & 0 \end{pmatrix}$.
The matrix algebra has only one simple module, the standard two-dimensional one, which we shall denote by $V$, even as a module over $H(\mu)$.
Its character $\chi_V$ maps $G \mapsto 0$, $G^2 \mapsto -2$ and $G^3 \mapsto 0$.
In total then, $H(\mu)$ has three simple modules $\kk^+_+$, $\kk^+_-$ and $V$ and the corresponding $\phs{i}$ of the character-projector formula \eqref{eq:character-projector-formula-nonssi} are $\ph = \frac{1}{4}(1+g+g^2+g^3)$ and $\phs{\kk^+_-} = \frac{1}{4}(1-g+g^2-g^3)$ for the two simple modules of $H(\mu)^+ \cong H_4$, and $\phs{V} = 2 \cdot \frac{1}{4}(2 - 2g^2) = 1 - g^2$ for the unique simple module of $H(\mu)^- \cong \Mat{2}{\kk}$.
Here we see now that $\phs{V}^2 = (1 - g^2)^2 = 1 - 2g^2 + g^4 = 2(1 - g^2)$ is not an idempotent.
The reason for this failure to be idempotent is, one could argue, the factor $2$ in $\phs{V} = 2 \cdot \frac{1}{4}(2 - 2g^2)$, which comes from the factor $\dim(V)$ in the character-projector formula $\phs{V} := \dim(V) \chi_V(S(\ph_{(1)})) \ph_{(2)}$.

Indeed, if we set $\mu = 0$, then $V$ is not a simple $H(\mu)$-module anymore, but rather fits in a short exact sequence
$ 0 \to \kk^-_- \to V \to \kk^-_+ \to 0$,
so that the character splits into $\chi_V = \chi_{\kk^-_-} + \chi_{\kk^-_+}$.
Therefore, for $\mu = 0$, $\phs{V}$ becomes $\dim(V) (S(\chi_{\kk^-_-} + \chi_{\kk^-_+})\ot \id{}) (\Delta(\ph)) = \dim(V) (\phs{\kk^-_-} + \phs{\kk^-_+})$.
The sum of orthogonal idempotents $\phs{\kk^-_-}$ and $\phs{\kk^-_+}$ is of course again an idempotent, so here we do not get an idempotent precisely because of the factor $\dim(V) = 2$.
\end{example}

\subsection{Idempotence of the conjectured idempotents} \label{subsec:idempotence-of-conjectured-idempotents}

Making use the Chevalley property we will now arrive at further results concerning the idempotent $\ph \in H$ and isotypic decompositions of $H$, towards proving our Conjecture \ref{conjecture}.

We reiterate that if $H$ has the Chevalley property, then $H/J$ is a semisimple Hopf algebra and the quotient map $\pi : H \lto H/J$ is a surjective morphism of Hopf algebras.
We want to determine the image $\pi(\ph)$ of the idempotent $\ph \in H$ under this surjection.

\begin{lemma} \label{lem:haarintegralofH/J}
The image of the element $\ph \in H$ under the quotient map $\pi : H \to H/J$ is $p_{H/J} := \frac{1}{\dim(H/J)} \chi_{(H/J)^*} \in H/J$, and hence by Proposition \ref{prop:haar-integral-given-by-character} the Haar integral  of the semisimple Hopf algebra $H/J$.
\end{lemma}
\begin{proof}
Considering $\ph$ as an element of $H^{**}$, $\pi(\ph) \in (H/J)^{**}$ is the restriction of $\ph : H^* \to \kk$ to the subalgebra $(H/J)^*$.
According to the Nichols-Zoeller theorem $H^*$ is free as an $(H/J)^*$-module, i.e.: $H^* \cong ((H/J)^*)^{\oplus N}$, for $N \in \NN$ such that $\dim(H)=N \dim(H/J)$.
This implies for the characters: $\chi_{H^*} \restrict{(H/J)^*} = N \chi_{(H/J)^*}$ and thus we have $\pi(\ph) = \frac{1}{\dim(H)} \chi_{H^*} \restrict{(H/J)^*} = \frac{N}{\dim(H)} \chi_{(H/J)^*} = \frac{1}{\dim(H/J)} \chi_{(H/J)^*} = p_{H/J}$.
\end{proof}

From this follows the main result of this subsection:

\begin{theorem} \label{thm:idempotents_for_one-dimensional_simples}
Let $H$ be a Hopf algebra with the Chevalley property and let $\chi : H \to \kk$ be the character of a non-zero one-dimensional (hence, simple) $H$-module.
Then the element $\phs{\chi} = (S(\chi)\otimes\id{H})(\Delta(p_H)) \in H$ is an idempotent such that $H \phs{\chi} \subseteq H$ is a $\chi$-isotypic component of $H$.
\end{theorem}
\begin{proof}
Since $\chi : H \to \kk$ is the character of a one-dimensional $H$-module, it is an algebra morphism.
This implies that the element $\phs{\chi} = (S(\chi)\otimes\id{H})(\Delta(p_H)) \in H$ is an idempotent as it is the image of the idempotent $\ph$ under an algebra morphism.
Furthermore $\chi: H \to k$ factors through $H/J$ as $\chi$ corresponds to a simple $H$-module and, hence, $\pi(\phs{\chi}) = \pi((S(\chi)\otimes\id{H})(\Delta(\ph))) = (S(\chi)\otimes\id{H})(\Delta(\pi(\ph))) = (S(\chi)\otimes\id{H})(\Delta(p_{H/J}))$, using Lemma \ref{lem:haarintegralofH/J} in the last step.

Since $p_{H/J}$ is the Haar integral of the semisimple Hopf algebra $H/J$, we thus know that $\pi(\phs{\chi}) \in H/J$ is the central idempotent projecting to the isotypic component of $\chi$ in $H/J$.

To conclude the proof we show that, if $\tild{e_S} \in H$ is an idempotent preimage of the central idempotent $e_S \in H/J$ corresponding to a simple $H$-module (and, hence, also $H/J$-module) $S$, then the submodule $H\tild{e_S} \subseteq H$ is a projective cover of $S^{\oplus \dim S}$.
Indeed, $H \tild{e_S}$ is a projective cover of $H \tild{e_S} / J(H \tild{e_S})$.
Moreover, note that we have the isomorphism of $H$-modules $H \tild{e_S} / J(H \tild{e_S}) \cong (H/J) e_S \cong S^{\oplus\dim(S)}$.
The first isomorphism follows from $J(H\tild{e_S}) = H\tild{e_S} \cap J(H) = \ker(\pi\restrict{H\tild{e_S}})$, where $\pi : H \lto H/J$ denotes the quotient map, together with the fact that $(H/J)e_S$ is the image of the restricted quotient map $\pi\restrict{H\tild{e_S}} : H\tild{e_S} \lto H/J$.
\end{proof}

In this subsection we have shown that, for a one-dimensional simple $H$-module $i \in I$, the conjectured idempotent $\phs{i}$ is indeed an idempotent projecting to an $i$-isotypic component of $H$.
However, we do not yet know whether these idempotents for one-dimensional simple modules are orthogonal to each other.

\subsection{Orthogonality of the conjectured idempotents} \label{subsec:the-phs-sum-up-to-one}

In particular, if we assume that $H$ has only one-dimensional simple modules, i.e. $H$ is \emph{basic},
then we know so far, combining Propositions \ref{prop:isotypic_decomposition} and \ref{thm:idempotents_for_one-dimensional_simples}, that $H = \bigoplus_{i \in I} H \phs{i}$ is an isotypic decomposition for $H$.
However, we do not yet know whether the natural projections of $\bigoplus_{i \in I} H \phs{i}$ onto the direct summands $H \phs{i}$ are the same as the projections given by right multiplication with the idempotents $\phs{i}$.
This is the case if and only if the $\phs{i}$ are orthogonal to each other, i.e. $\phs{i} \phs{j} = \delta_{i,j} \phs{i}$ for all $i, j \in I$.
In this subsection we prove a result which implies in particular for a basic Hopf algebra $H$, using our results from Subsection \ref{subsec:idempotence-of-conjectured-idempotents}, that under a certain additional assumption they are. 

Due to the following lemma, showing that $\sum_{i\in I} \phs{i} = 1$ is sufficient to show that the idempotents $\phs{i}$ are pairwise orthogonal to each other.
\begin{lemma} \label{lem:partitionofunityimpliesorthogonality}
Let $H$ be an algebra with decomposition $H = \bigoplus_i H_i$ into left $H$-submodules $H_i$ and let $p_i \in H_i$ be elements such that $\sum_i p_i = 1$.
Then $p_i p_j = \delta_{i,j} p_i$ for all $i, j$ and $H_i = H p_i$.
\end{lemma}
\begin{proof}
We have for any $i$ that $p_i = p_i 1 = \sum_j p_i p_j$.
Then $p_i \in H_i$ by assumption and $p_i p_j \in H_j$ because $H_j$ is an $H$-submodule together imply that $p_i p_j = \delta_{i,j} p_i$ using the direct sum property of $\bigoplus_i H_i$.

It is left to show that $H_i = H p_i$.
$H p_i \subseteq H_i$ follows immediately from the facts that $p_i \in H_i$ and that $H_i$ is an $H$-submodule.
In order to show that also $H_i \subseteq H p_i$, assume that $h_i \in H_i$.
We have $h_i = h_i 1 = \sum_j h_i p_j$.
Since $h_i p_j \in H_j$, this implies that $h_i = h_i p_i$, concluding the proof that $H_i \subseteq H p_i$.
\end{proof}

In this subsection we therefore want to show that $\sum_{i\in I} \phs{i} = 1$ (Thm. \ref{thm:the-phs-sum-up-to-one}).

First we need a lemma.
Note that the regular character $\chi_H : H \lto \kk$ lifts to $H/J$ via the quotient map $\pi : H \lto H/J$ (as all characters of $H$ do, since the Jacobson radical $J$ is a nil ideal).
In fact, we can furthermore show that on $H/J$ it is proportional to the character $\chi_{H/J}$ of the regular $H/J$-module, i.e.\ we have:
\begin{lemma} \label{lem:regular-character-chevalley}
Let $H$ be a Hopf algebra with the Chevalley property.
Then: $\chi_H = \frac{\dim(H)}{\dim(H/J)} \chi_{H/J} \circ \pi$.
\end{lemma}
\begin{proof}
On the hand, we have for any left $H$-module $M$ a canonical $H$-module isomorphism $H \ot M \cong H \ot M_\text{triv}$.
Applying this to $M = \pi^*(H/J)$, by which we denote $H/J$ with the action of $H$ via the quotient map $\pi: H \lto H/J$, we obtain the equality of characters $\chi_H \cdot \chi_{\pi^*(H/J)} = \chi_H \dim(H/J) \in H^*$.

On the other hand, for any $H/J$-module $N$ we have a canonical isomorphism of $H/J$-modules $N \ot H/J \cong N_\text{triv} \ot H/J$, which implies for the characters: $\chi_N \cdot \chi_{H/J} = \dim(N) \chi_{H/J}$.

Next, observe that $\pi:H \lto H/J$ induces an isomorphism $\pi^* : G_0(H/J) \lto G_0(H)$ of the Grothendieck rings of $\lmod{H}$ and $\lmod{(H/J)}$.
Since the character of a module only depends on its class in the Grothendieck ring, this implies that there exists an $H/J$-module $V$ such that $\chi_H = \pi^*{\chi_V}$.
Moreover, $\dim(V) = \dim(H)$, since modules in the same class in the Grothendieck ring have the same dimension.

In summary, we obtain
\begin{align*}
\chi_H \cdot \chi_{\pi^*(H/J)} &= \pi^*(\chi_V \cdot \chi_{H/J}) \\
&= \pi^*(\dim(V) \chi_{H/J}) \\
&= \dim(H) \pi^*(\chi_{H/J}).
\end{align*}
Together with the first paragraph of the proof this shows the claim.
\end{proof}

As always denote by $I$ the set of isomorphism classes of simple $H$-modules and for $i \in I$ write $\phs{i} := \frac{\dim(i)}{\dim(H)} (S(\chi_i) \ot \id{H})(\Delta(\chi_{H^*}))$, as in Conjecture \ref{conjecture}, where $\chi_i \in H^*$ is the character of the simple $H$-module $S_i$ and where $\chi_{H^*} \in H^{**} \cong H$ is the regular character of $H^*$.

The following Theorem \ref{thm:the-phs-sum-up-to-one} proves that $\sum_{i\in I} \phs{i} = 1$ holds for a Hopf algebra $H$ with the Chevalley property, under an additional assumption on $H$.
In order to formulate this assumption, we have to introduce the so-called \emph{Hecke algebra} associated to $H^*$.
Since $H$ has the Chevalley property, $H/J$ is its maximal semisimple quotient-Hopf-algebra.
Dually this means that $H^*_0 := (H/J)^* \subseteq H^*$ is the maximal semisimple sub-Hopf-algebra of $H^*$.
(In other words, $H^*_0 = (H/J)^*$ is in particular the coradical \cite{montgomery} of $H^*$.)
Hence we can consider the unique Haar integral $\Lambda_0 \in H^*_0$ of this semisimple Hopf algebra $H^*_0$.
Now the space $\Lambda_0 H^* \Lambda_0 \subseteq H^*$ is an (in general, not unital) subalgebra of $H^*$ with unit $\Lambda_0$.
It can also be characterised as the endomorphism algebra $\End{H^*}{H^* \Lambda_0} \cong \Lambda_0 H^* \Lambda_0$ of the $H^*$-module $H^* \Lambda_0$ induced from the trivial $H^*_0$-module along the inclusion $H^*_0 \subseteq H^*$.
Hence:
\begin{definition} \label{def:hecke-algebra}  
We call the algebra $\Lambda_0 H^* \Lambda_0$ with unit $\Lambda_0$ the \emph{Hecke algebra} $\mathscr{H}(H^*,H^*_0)$ associated to the trivial representation of $H^*_0 \subseteq H^*$.
\end{definition}
Now we can state our result.

\begin{theorem} \label{thm:the-phs-sum-up-to-one}
Let $H$ be a Hopf algebra with the Chevalley property.
Let $\Lambda_0 \in H^*$ be the Haar integral of the maximal semisimple sub-Hopf-algebra $H^*_0 = (H/J)^*$.

Then
\[ \sum_{i\in I} \phs{i} = 1_H \]
if and only if the Hecke algebra $\Lambda_0 H^* \Lambda_0$ has up to isomorphism only one simple module.
\end{theorem}
\begin{proof}
Since for the regular character $\chi_H \in H^*$ of $H$ we have by Lemma \ref{lem:regular-character-chevalley} \[ \chi_H = \frac{\dim(H)}{\dim(H/J)} \pi^*(\chi_{H/J}) = \frac{\dim(H)}{\dim(H/J)} \sum_{i\in I} \dim(i) \chi_i = \frac{\dim(H)}{\dim(H/J)} \sum_{i\in I} \dim(i) S(\chi_i), \]
the equation $\sum_{i\in I} \phs{i} = 1_H$ is equivalent to
\[ \frac{\dim(H/J)}{\dim(H)^2} (\chi_H \ot \id{H})(\Delta(\chi_{H^*})) = 1_H . \]
Using that $\Lambda_0 = \frac{1}{\dim(H/J)} \chi_{H/J}$ by semisimplicity of the Hopf algebra $H/J$, and $\frac{1}{\dim(H/J)} \chi_{H/J} = \frac{1}{\dim(H)} \chi_{H}$ by Lemma \ref{lem:regular-character-chevalley}, we rewrite this equation to
\[ (\Lambda_0 \ot \id{H})(\Delta(\chi_{H^*})) = \frac{\dim(H)}{\dim(H/J)} 1_H , \]
which can be rewritten as
\begin{equation} \label{eq:completenessrelationofidempotents-rewritten}
\chi_{H^*}(\Lambda_0 \cdot -) = \frac{\dim(H)}{\dim(H/J)} \eps_{H^*} .
\end{equation}
Since the subalgebra $H^*_0$ is semisimple, we can decompose $H^*$ as an $H^*_0$-bimodule as
\[ H^* = \bigoplus_{i,j \in I'} e_i H^* e_j =: \bigoplus_{i,j \in I'} H^*_{i,j} , \]
where $(e_i)_{i\in I'}$ are the central orthogonal idempotents of the semisimple algebra $H^*_0$  (in particular, $e_\II = \Lambda_0$, where $e_\II$ is the idempotent corresponding to the trivial $H^*_0$-module).
Therefore, with respect to this decomposition of $H^*$ we have:
\[ H^*_{i,j} \cdot H^*_{k,l} \subseteq \begin{cases} H^*_{i,l} &:\quad  j=k, \\ 0 &:\quad j\neq k. \end{cases} \]
In particular, if $i\neq j$, then $H^*_{i,j}$ contains only nilpotent elements.
From this it follows that both sides of equation \eqref{eq:completenessrelationofidempotents-rewritten} vanish on \[ \bigoplus_{\substack{i,j \in I' \\ (i,j)\neq (\II,\II)}} H^*_{i,j} \]
Indeed, both $\chi_{H^*}$ (being a character) and $\eps_{H^*}$ (being an algebra map) vanish on nilpotent elements of $H^*$.
Furthermore, for $i\neq \II$, $\chi_{H^*}(\Lambda_0\cdot -)$ vanishes on $H^*_{i,j}$ by orthogonality of $(e_i)_{i\in I'}$ and so does $\eps_{H^*}$ for the same reason, since $\eps_{H^*}(\Lambda_0)=1$.

Therefore, equation \eqref{eq:completenessrelationofidempotents-rewritten} is equivalent to
\begin{equation} \label{eq:completenessrelationofidempotents-reduced}
\chi_{H^*} \restrict{\Lambda_0 H^* \Lambda_0} = \frac{\dim(H)}{\dim(H/J)} \eps_{H^*} \restrict{\Lambda_0 H^* \Lambda_0} .
\end{equation}
For this, note that left multiplication by $\Lambda_0 H^* \Lambda_0$ on $H^*$ is non-zero only on the direct summand $\Lambda_0 H^* \subseteq \bigoplus_{i\in I'} e_i H^* = H^*$.
This defines an action of the algebra $\Lambda_0 H^* \Lambda_0$ (with unit $\Lambda_0$) on  $\Lambda_0 H^*$.
Thus equation \eqref{eq:completenessrelationofidempotents-reduced} is equivalent to the statement that the character of $\Lambda_0 H^*$ as a left $\Lambda_0 H^* \Lambda_0$-module is equal to $\frac{\dim(H)}{\dim(H/J)} \eps_{H^*}\restrict{\Lambda_0 H^* \Lambda_0}$.
We can show this to be equivalent to the statement that up to isomorphism, the algebra $\Lambda_0 H^* \Lambda_0$ has only one simple module:
the trivial one defined on $\kk$ via ${\eps_{H^*} \restrict{\Lambda_0 H^* \Lambda_0} : \Lambda_0 H^* \Lambda_0 \lto \kk}$.

Indeed, if this is the case, then the character of the $\Lambda_0 H^* \Lambda_0$-module $\Lambda_0 H^*$ must be equal to $n \cdot \eps_{H^*} \restrict{\Lambda_0 H^* \Lambda_0}$, where $n \in \NN$ is the length of the Jordan-Hölder series of the module $\Lambda_0 H^*$.
Evaluating on $\Lambda_0$, which is the unit for the algebra $\Lambda_0 H^* \Lambda_0$, gives $n = \dim(\Lambda_0 H^*)$.
Therefore, we obtain $\chi_{H^*} \restrict{\Lambda_0 H^* \Lambda_0} = \dim(\Lambda_0 H^*) \eps_{H^*} \restrict{\Lambda_0 H^* \Lambda_0}$.
It remains to verify that $\dim(\Lambda_0 H^*) = \frac{\dim(H)}{\dim(H/J)}$.
Indeed, by Nichols-Zoeller $H^* \cong (H^*_0)^N$ as a left $H^*_0$-module, for $N = \frac{\dim(H)}{\dim(H/J)}$.
Under this isomorphism we have $\Lambda_0 H^* \cong (\Lambda_0 H^*_0)^N = (\Lambda_0 \kk)^N$, since $\Lambda_0$ is the Haar integral of $H^*_0$.
Hence, $\dim(\Lambda_0 H^*) = N = \frac{\dim(H)}{\dim(H/J)}$.

Conversely, if there is another simple $\Lambda_0 H^* \Lambda_0$-module, not isomorphic to the trivial one given by $\eps_{H^*} \restrict{\Lambda_0 H^* \Lambda_0}$, then it is also a quotient of the regular $\Lambda_0 H^* \Lambda_0$-module and, hence, of $\Lambda_0 H^*$.
But then the character of $\Lambda_0 H^*$ cannot be equal to $\frac{\dim(H)}{\dim(H/J)} \eps_{H^*}\restrict{\Lambda_0 H^* \Lambda_0}$.
\end{proof}

Finally, we conclude from Theorems \ref{thm:idempotents_for_one-dimensional_simples} and \ref{thm:the-phs-sum-up-to-one} the validity of Conjecture \ref{conjecture} for a certain subclass of the Hopf algebras with the Chevalley property:

\begin{corollary} \label{corollary}
Let $H$ be a finite-dimensional basic Hopf algebra over $\kk$ and denote by $H^*_0 := (H/J)^*$ the maximal semisimple sub-Hopf-algebra of its dual $H^*$.
Assume that the associated Hecke algebra $\mathscr{H}(H^*, H^*_0)$ (cf. Definition \ref{def:hecke-algebra}) has, up to isomorphism, a unique simple $\mathscr{H}(H^*, H^*_0)$-module.
Then Conjecture \ref{conjecture} holds for $H$, i.e. $(\phs{i} = \dim(S_i) \chi_i(S(\ph_{(1)})) \ph_{(2)})_{i \in I}$ are orthogonal idempotents such that $H = \bigoplus_{i \in I} H \phs{i}$ is an isotypic decomposition for $H$.
\end{corollary}
\begin{proof}
Theorem \ref{thm:idempotents_for_one-dimensional_simples} and Proposition \ref{prop:isotypic_decomposition} imply that the $(\phs{i})_{i \in I}$ are idempotents and that $H = \bigoplus_{i \in I} H \phs{i}$ is an isotypic decomposition, since $H$ has only one-dimensional simple $H$-modules.
Furthermore, Theorem \ref{thm:the-phs-sum-up-to-one} and Lemma \ref{lem:partitionofunityimpliesorthogonality} together imply that the $(\phs{i})_{i \in I}$ are orthogonal.
\end{proof}

\subsection{Hopf algebras with the Chevalley property and the dual Chevalley property}

Let $H$ be a Hopf algebra over $\kk$ with both the Chevalley property and the \emph{dual Chevalley property} (i.e. also the dual Hopf algebra $H^*$ has the Chevalley property).
Then a lot more can be said about the structure of $H$ and, in particular, our conjecture that the elements $\phs{i}$ give an isotypic decomposition of $H$ can be verified.

\begin{lemma}   \label{lem:retract-for-chevalley-and-co-chevalley}
Let $H$ be a Hopf algebra over $\kk$ with both the Chevalley property and the dual Chevalley property.
Then there exists a Hopf algebra section $\iota : H/J \lto H$ of the quotient map $\pi : H \lto H/J$, identifying the maximal semisimple quotient Hopf algebra $H/J$ with the maximal semisimple sub-Hopf-algebra $(H^* / J_{H^*})^* \subseteq H$.
\end{lemma}
\begin{remark}
By Radford's projection theorem this implies that $H$ is isomorphic to the Radford biproduct $R \# (H/J)$, where $R := H^{\co (H/J)}$, the subspace of right $(H/J)$-coinvariants of $H$.
\end{remark}
\begin{proof}
Since $H^*$ has the Chevalley property, its maximal semisimple Hopf algebra quotient is ${H^* \twoheadrightarrow H^* / J_{H^*}}$, where $J_{H^*}$ is the Jacobson radical of $H^*$.
This means that \[ ( H^* / J_{H^*} )^* \subseteq H^{**} \cong H \] is the maximal semisimple sub-Hopf-algebra of $H$.
Consider the composition $( H^* / J_{H^*} )^* \hookrightarrow H \twoheadrightarrow H/J$ of inclusion and quotient map.
We will show that it is an isomorphism of Hopf algebras $( H^* / J_{H^*} )^* \cong H/J$.

Firstly, it is injective because $( H^* / J_{H^*} )^* \cap J = 0$, because $( H^* / J_{H^*} )^* \cap J \subseteq ( H^* / J_{H^*} )^*$ is a nilpotent ideal of $( H^* / J_{H^*} )^*$, but $( H^* / J_{H^*} )^*$ is semisimple and therefore has no non-zero nilpotent ideal.

Secondly, we see that $( H^* / J_{H^*} )^* \hookrightarrow H \twoheadrightarrow H/J$ is also surjective, because its dual is \[ H^*_0 \eqbydef (H/J)^* \hookrightarrow H^* \twoheadrightarrow H^* / J_{H^*}, \] which is just the inclusion of the maximal semisimple sub-Hopf-algebra followed by the surjection to the maximal semisimple quotient Hopf algebra for the dual Hopf algebra $H^*$, and this is injective by the above argument.
\end{proof}

\begin{proposition} \label{prop:chevalley-and-cochevalley}
Let $H$ be a Hopf algebra over $\kk$ with both the Chevalley property and the dual Chevalley property.
Then the family $\phs{i} \in H$, as described in Conjecture \ref{conjecture}, gives a set of orthogonal idempotents of an isotypic decomposition of $H$.
\end{proposition}
\begin{proof}
We will prove this by proving that $\phs{i} = \iota(e_i)$ for the Hopf algebra inclusion ${\iota : H/J \lto H}$, which we have shown to exist in Lemma \ref{lem:retract-for-chevalley-and-co-chevalley}, where $e_i \in H/J$ are the central orthogonal idempotents of the isotypic decomposition of the semisimple Hopf algebra $H/J$.
We have $\phs{i} = \frac{\dim(S_i)}{\dim(H)} \chi_i(S({\chi_{H^*}}_{(1)})) {\chi_{H^*}}_{(2)}$ and $e_i = \frac{\dim(S_i)}{\dim(H/J)} \tild{\chi_i}(S({\chi_{(H/J)^*}}_{(1)})) {\chi_{(H/J)^*}}_{(2)}$, where $\chi_i \in H^*$ is the character of the $i$-th simple $H$-module and $\tild{\chi_i} \in (H/J)^*$ is the character of the corresponding $H/J$-module, i.e. $\chi_i = \tild{\chi_i} \circ \pi$.

What we thus have to show is that \[ \frac{1}{\dim(H)} \chi_i(S({\chi_{H^*}}_{(1)})) {\chi_{H^*}}_{(2)} = \frac{1}{\dim(H/J)} \tild{\chi_i}(S({\chi_{(H/J)^*}}_{(1)})) \iota( {\chi_{(H/J)^*}}_{(2)} )	.	\]
Using $\pi \circ \iota = \id{H/J}$ and that $\iota$ is a morphism of Hopf algebras, we obtain
\begin{align*}
\frac{1}{\dim(H/J)} \tild{\chi_i}(S({\chi_{(H/J)^*}}_{(1)})) \iota( {\chi_{(H/J)^*}}_{(2)} )
&= \frac{1}{\dim(H/J)} \tild{\chi_i}(\pi(\iota(S({\chi_{(H/J)^*}}_{(1)})))) \iota( {\chi_{(H/J)^*}}_{(2)} ) \\
&= \frac{1}{\dim(H/J)} \chi_i(\iota(S({\chi_{(H/J)^*}}_{(1)}))) \iota( {\chi_{(H/J)^*}}_{(2)} ) \\
&= \frac{1}{\dim(H/J)} \chi_i(S({\iota(\chi_{(H/J)^*})}_{(1)})) {\iota(\chi_{(H/J)^*})}_{(2)} .
\end{align*}
Thus it is left to show that
\[ \frac{1}{\dim(H/J)} \iota(\chi_{(H/J)^*}) = \frac{1}{\dim(H)} \chi_{H^*} \]
Denote by $\Pi : H^* \lto H^* / J_{H^*}$ the quotient map, which is a morphism of Hopf algebras due to the Chevalley property of $H^*$.
Then by Lemma \ref{lem:regular-character-chevalley} applied to $H^*$, we have \[ \frac{1}{\dim(H)} \chi_{H^*} = \frac{1}{\dim(H^* / J_{H^*})} \Pi^*(\chi_{H^* / J_{H^*}}) . \]
But in Lemma \ref{lem:retract-for-chevalley-and-co-chevalley} we have identified the Hopf algebras $(H^* / J_{H^*})^*$ and $H/J$ and via this identification, by definition, the injection $\iota : H/J \lto H$ corresponds to 
\[ \Pi^* : (H^* / J_{H^*})^* \lto H^{**} \cong H. \]
Hence, $\frac{1}{\dim(H^* / J_{H^*})} \Pi^*(\chi_{H^* / J_{H^*}}) = \frac{1}{\dim(H/J)} \iota(\chi_{(H/J)^*}),$
concluding the proof.
\end{proof}

\section{Examples}  \label{sec:examples}

In this last section of the paper we discuss two examples of Hopf algebras with the Chevalley property that we can show to satisfy our conjecture.
The first example is a basic Hopf algebra for which we can verify the assumptions of Theorem \ref{thm:the-phs-sum-up-to-one}, which implies that the conjecture holds.
The second example is not basic and it does not follow directly from our general results that it satisfies our conjecture, but we carry out explicit computations to show that it does.

\subsection{The dual of a deformation of the double cover of Sweedler's Hopf algebra: a basic Hopf algebra} \label{subsec:example-basic-hopf-algebra}

Recall the $8$-dimensional Hopf algebra $H(\mu)$ from Example \ref{exa:non-chevalley}.
Without loss of generality let us set $\mu = 2$, since for all $\mu \in \kk^\times$,
$H(\mu)$ is in the same isomorphism class of Hopf algebras.
Here we are interested in its dual Hopf algebra, which does satisfy the Chevalley property, and which by slight abuse of notation we will denote by $H$, so that $H^* = H(2)$.
Recall that $H^*$ is generated as an algebra by $g$ and $x$ subject to the relations
\begin{align*}
g^4 &= 1 \\
g x g^{-1} &= -x \\
x^2 &= (1 - g^2),
\end{align*}
with the co-multiplication given by $\Delta (g) = g \ot g$ and $\Delta (x) = x \ot 1 + g \ot x$.

$H^*$ has a $\ZZ_2$-grading as an algebra, $H^* = (H^*)_0 \oplus (H^*)_1$, where $(H^*)_0 = \linhuell_\kk\{1,g,g^2,g^3\} = \kk G$ is the group algebra of the group $G$ of group-like elements and $(H^*)_1 = (H^*)_0 \cdot x$.
Furthermore, we have $\Delta (H^*)_0 \subseteq (H^*)_0 \ot (H^*)_0$ and $\Delta (H^*)_1 \subseteq (H^*)_1 \ot (H^*)_0 \oplus (H^*)_0 \ot (H^*)_1$.

The simple $H^*$-comodules are given by the simple comodules of the coradical of $H^*$, which is $(H^*)_0 = \kk G \subseteq H^*$, and therefore we have four one-dimensional simple $H^*$-comodules corresponding to the four group elements of $G$.

For the dual Hopf algebra $H$ this means that there are four simple $H$-modules and they are one-dimensional and given by the four group elements $1,g,g^2,g^3 \in G$ interpreted as elements of $H^*$.

Let us consider the corresponding four elements $\ph, \phs{g}, \phs{g^2}, \phs{g^3} \in H$ given by the character-projector formula \eqref{eq:character-projector-formula-nonssi}.
Since all four simple modules are one-dimensional, our results from Subsection \ref{subsec:idempotence-of-conjectured-idempotents} imply that these four elements are idempotents projecting to appropriate isotypic components.
The remaining question of their orthogonality is answered affirmatively by Theorem \ref{thm:the-phs-sum-up-to-one}.
Indeed, we can verify that the Hecke algebra $\Lambda_0 H^* \Lambda_0$, where here $\Lambda_0 = \frac{1}{4} (1+g+g^2+g^3) \in H^*_0$, satisfies the condition of Theorem \ref{thm:the-phs-sum-up-to-one}, since
\[ \Lambda_0 H^* \Lambda_0 = \linhuell_\kk\{ \Lambda_0, \Lambda_0 x \Lambda_0 \} \cong \kk[x] / (x^2) .  \]

Finally we can also compute the orthogonal idempotents $\ph, \phs{g}, \phs{g^2}, \phs{g^3} \in H$ explicitly:
With respect to the basis $\{ 1, g, g^2, g^3, x, gx, g^2 x, g^3 x \}$ for $H^*$, with dual basis \[ \{ \delta_1, \delta_g, \delta_{g^2}, \delta_{g^3}, \delta_x, \delta_{gx}, \delta_{g^2 x}, \delta_{g^3 x} \} \] for $H$, it is easy to see by computation that we have:
\begin{align*}
\ph &= \delta_1 \\
\phs{g} &= \delta_g \\
\phs{g^2} &= \delta_{g^2} \\
\phs{g^3} &= \delta_{g^3}
\end{align*}
Therefore, indeed, the idempotents are orthogonal to each other.

\subsection{A Hopf algebra with the Chevalley property which is not basic} \label{subsec:example-72-dim}

Finally we consider an example of a Hopf algebra with the Chevalley property that is not basic, i.e. it does not have the property that all its simple modules are one-dimensional.
We show by computation that it still satisfies Conjecture \ref{conjecture}.
This gives evidence that our conjecture holds for all Hopf algebras with the Chevalley property, even though our general results only cover basic Hopf algebras (with an additional assumption on the associated Hecke algebra).

We describe the Hopf algebra $H$ that we want to consider, first by its dual Hopf algebra $H^*$.
As an algebra, $H^*$ is generated by the elements $a, b, c$ and $\{ e_g \mid g \in S_3 \}$ subject to the following relations:
\begin{align*}
e_g e_h &= \delta_{g,h} e_g \quad \forall g, h \in S_3 \\
\sum_{g\in S_3} e_g &= 1 \\
a e_g &= e_{(12)g} a \quad \forall g \in S_3 \\
b e_g &= e_{(23)g} b \quad \forall g \in S_3 \\
c e_g &= e_{(31)g} c \quad \forall g \in S_3 \\
ab + bc + ca &= 0 \\
ac + cb + ba &= 0 \\
a^2 &= \lambda_{ab}(e_{13} + e_{132}) + \lambda_{ac}(e_{23} + e_{123}) \\
b^2 &= \lambda_{bc}(e_{12} + e_{132}) + \lambda_{ba}(e_{13} + e_{123}) \\
c^2 &= \lambda_{ca}(e_{23} + e_{132}) + \lambda_{cb}(e_{12} + e_{123}) ,
\end{align*}
where $\lambda_{ab}, \lambda_{bc}, \lambda_{ca} \in \kk$ are the deformation parameters.
The co-multiplication on $H^*$ is given by
\begin{align*}
\Delta(e_g) &= \sum_{h \in S_3} e_{g h^{-1}} \ot e_h \quad \forall g \in S_3 \\
\Delta(a) &= a\ot 1  + (e_1-e_{12})\ot a + (e_{132}-e_{13})\ot b + (e_{123}-e_{23})\ot c	\\
\Delta(b) &= b\ot 1  + (e_1-e_{23})\ot b + (e_{132}-e_{12})\ot c + (e_{123}-e_{13})\ot a	\\
\Delta(c) &= c\ot 1  + (e_1-e_{13})\ot c + (e_{132}-e_{23})\ot a + (e_{123}-e_{12})\ot b .
\end{align*}
We see that $H^*$ contains the dual group algebra $\kk^{S_3} = \linhuell_\kk\{e_g \mid g\in S_3\}$ of the symmetric group $S_3$ as a sub-Hopf-algebra.
In fact, $H^*$ is a cocycle deformation of the Radford biproduct $B \# \kk^{S_3}$, where $B$ is the Nichols algebra generated by $a, b$ and $c$ and the dual group algebra $\kk^{S_3}$ is the maximal semisimple sub-Hopf-algebra.
For more details about this Hopf algebra see \cite{meir, andVay}.

For the dual Hopf algebra $H$ with the Chevalley property this means that it possesses the group algebra $\kk S_3$ as its maximal semisimple quotient Hopf algebra and, hence, the simple $H$-modules are given by the irreducible representations of $S_3$: the trivial representation, the sign representation and the two-dimensional simple $H$-module $V$, which is defined by:
\begin{align*}
(12) &\lmapsto \begin{pmatrix} 0 & 1 \\ 1 & 0 \end{pmatrix} \\
(123) &\lmapsto \begin{pmatrix} 0 & -1 \\ 1 & -1 \end{pmatrix} .
\end{align*}
For later reference we note that the character $\chi_V \in H^*$ of this $H$-module is $\chi_V = 2 e_1 - e_{123} - e_{132}$.

Now we want to compute the three corresponding elements $\ph, \phs{\sgn}$ and $\phs{V} \in H$ from the conjectured character-projector formula \eqref{eq:character-projector-formula-nonssi}.

Let us start with $\ph = \frac{1}{\dim(H)} \chi_{H^*}$, that is we have to compute the trace of the regular representation of $H^*$.
A convenient basis for $H^*$ as a vector space is given by \[ \{ 1,a,b,c,ab,bc,ac,cb,aba,abc,bac,abac \} \times \{ e_g \mid g \in S_3 \} . \]
Hence we can think of $H^*$ as being $\NN_{\geq 0}$-graded 
as a coalgebra (the grading comes from the grading of the Nichols algebra $B$), where the degree is determined by the length of the word in the letters $a, b$ and $c$.

The basis elements $e_g, g \in S_3,$ of degree $0$ are idempotents with $12$-dimensional image, that is $\chi_{H^*}(e_g)=12$.
This determines the trace of all elements of degree $0$.

The basis elements of degree $1$ are nilpotent:
For example, $$a e_g a e_g = a^2 e_{(12) g} e_g = a^2 \delta_{(12)g,g} e_g = 0$$ for all $g\in S_3$ since $(12) g \neq g$.
This implies that $\chi_{H^*}(a e_g) = 0$ and likewise $\chi_{H^*}(b e_g) = 0$ and $\chi_{H^*}(c e_g) = 0$ for all $g\in S_3$.

A similar argument shows that the trace is zero on the degree $2$ and degree $3$ parts of $H^*$.
On the degree $4$ component this argument does not work anymore, since $(31) (12) (23) (12) = 1$ and hence
\begin{equation} abac e_g abac e_g = (abac)^2 e_g, \label{abacegabaceg} \end{equation}
which we cannot immediately see to be zero by orthogonality of the $(e_g)_{g\in S_3}$ as before.
In order to obtain an explicit expression for $\chi_{H^*}(abac e_g)$, we start by computing $(abac)^2$.
\begin{lemma} \label{lem:abacabac}
\[ (abac)^2 =  abac((\lambda_{ab}^2 + \lambda_{ca}^2 - \lambda_{bc}^2)e_{23} + 2\lambda_{ab}\lambda_{ac}e_{132}) - \lambda_{ab}^2\lambda_{ac}^2 e_{23} - \lambda_{ab}^2\lambda_{ac}^2  e_{132} \]
\end{lemma}
\begin{proof}
We straightforwardly calculate using the relations of the algebra $H^*$:
\begin{align}
abacabac &= - aba(ab + bc)bac \notag\\
&= - abaabbac - ababcbac \notag\\
&= - ab(\lambda_{ab}(e_{13} + e_{132}) + \lambda_{ac}(e_{23} + e_{123}))(\lambda_{bc}(e_{12} + e_{132}) + \lambda_{ba}(e_{13} + e_{123}))ac - ababcbac .
\end{align}
We further calculate
\begin{align}
ababcbac &= - ababc(ac + cb)c \notag\\
&= - ababcacc - ababccbc \notag\\
&= - ababca(\lambda_{ca}(e_{23} + e_{132}) + \lambda_{cb}(e_{12} + e_{123}))
- abab(\lambda_{ca}(e_{23} + e_{132}) + \lambda_{cb}(e_{12} + e_{123}))bc \notag\\
&= - ababca(\lambda_{ca}(e_{23} + e_{132}) + \lambda_{cb}(e_{12} + e_{123})) \notag\\
&\phantom{xxx}- ababbc(\lambda_{ca}(e_{(31)(23)(23)} + e_{(31)(23)(132)}) + \lambda_{cb}(e_{(31)(23)(12)} + e_{(31)(23)(123)})) \notag\\
&= - ababca(\lambda_{ca}(e_{23} + e_{132}) + \lambda_{cb}(e_{12} + e_{123})) \notag\\
&\phantom{xxx}- ababbc(\lambda_{ca}(e_{31} + e_{123}) + \lambda_{cb}(e_{23} + e_{1})) \notag\\
&= - ababca(\lambda_{ca}(e_{23} + e_{132}) + \lambda_{cb}(e_{12} + e_{123})) \notag\\
&\phantom{xxx}- aba(\lambda_{bc}(e_{12} + e_{132}) + \lambda_{ba}(e_{13} + e_{123}))c(\lambda_{ca}(e_{31} + e_{123}) + \lambda_{cb}(e_{23} + e_{1})) \notag\\
&= - ababca(\lambda_{ca}(e_{23} + e_{132}) + \lambda_{cb}(e_{12} + e_{123})) \notag\\
&\phantom{xxx}- abac (\lambda_{bc}(e_{(31)(12)} + e_{(31)(132)}) + \lambda_{ba}(e_{(31)(13)} + e_{(31)(123)})) (\lambda_{ca}(e_{31} + e_{123}) + \lambda_{cb}(e_{23} + e_{1})) \notag\\
&= - ababca(\lambda_{ca}(e_{23} + e_{132}) + \lambda_{cb}(e_{12} + e_{123})) \notag\\
&\phantom{xxx}- abac (\lambda_{bc}(e_{123} + e_{23}) + \lambda_{ba}(e_{1} + e_{12})) (\lambda_{ca}(e_{31} + e_{123}) + \lambda_{cb}(e_{23} + e_{1})) \label{ababcbac}
\end{align}
and
\begin{align}
ababca &= - aba(ab + ca)a \notag\\
&= - abaaba - abacaa \notag\\
&= - ab(\lambda_{ab}(e_{13} + e_{132}) + \lambda_{ac}(e_{23} + e_{123}))ba - abac(\lambda_{ab}(e_{13} + e_{132}) + \lambda_{ac}(e_{23} + e_{123})) \notag\\
&= - abba(\lambda_{ab}(e_{(12)(23)(13)} + e_{(12)(23)(132)}) + \lambda_{ac}(e_{(12)(23)(23)} + e_{(12)(23)(123)})) \notag\\
&\phantom{xxx}- abac(\lambda_{ab}(e_{13} + e_{132}) + \lambda_{ac}(e_{23} + e_{123})) \notag\\
&= - abba(\lambda_{ab}(e_{23} + e_{1}) + \lambda_{ac}(e_{12} + e_{132})) \notag\\
&\phantom{xxx}- abac(\lambda_{ab}(e_{13} + e_{132}) + \lambda_{ac}(e_{23} + e_{123})) \notag\\
&= - a(\lambda_{bc}(e_{12} + e_{132}) + \lambda_{ba}(e_{13} + e_{123}))a(\lambda_{ab}(e_{23} + e_{1}) + \lambda_{ac}(e_{12} + e_{132})) \notag\\
&\phantom{xxx}- abac(\lambda_{ab}(e_{13} + e_{132}) + \lambda_{ac}(e_{23} + e_{123})) \notag\\
&= - aa(\lambda_{bc}(e_{(12)(12)} + e_{(12)(132)}) + \lambda_{ba}(e_{(12)(13)} + e_{(12)(123)}))
(\lambda_{ab}(e_{23} + e_{1}) + \lambda_{ac}(e_{12} + e_{132})) \notag\\
&\phantom{xxx}- abac(\lambda_{ab}(e_{13} + e_{132}) + \lambda_{ac}(e_{23} + e_{123})) \notag\\
&= - aa(\lambda_{bc}(e_{1} + e_{13}) + \lambda_{ba}(e_{132} + e_{23}))
(\lambda_{ab}(e_{23} + e_{1}) + \lambda_{ac}(e_{12} + e_{132})) \notag\\
&\phantom{xxx}- abac(\lambda_{ab}(e_{13} + e_{132}) + \lambda_{ac}(e_{23} + e_{123})) \notag\\
&= - (\lambda_{ab}(e_{13} + e_{132}) + \lambda_{ac}(e_{23} + e_{123}))(\lambda_{bc}(e_{1} + e_{13}) + \lambda_{ba}(e_{132} + e_{23})) \notag\\
&\phantom{xxxxx}\cdot(\lambda_{ab}(e_{23} + e_{1}) + \lambda_{ac}(e_{12} + e_{132})) \notag\\
&\phantom{xxx}- abac(\lambda_{ab}(e_{13} + e_{132}) + \lambda_{ac}(e_{23} + e_{123})) \notag\\
&= - (\lambda_{ab}\lambda_{bc} e_{13} + \lambda_{ab}\lambda_{ba}  e_{132}
+ \lambda_{ac}\lambda_{ba} e_{23}) (\lambda_{ab}(e_{23} + e_{1}) + \lambda_{ac}(e_{12} + e_{132})) \notag\\
&\phantom{xxx}- abac(\lambda_{ab}(e_{13} + e_{132}) + \lambda_{ac}(e_{23} + e_{123})) \notag\\
&= - \lambda_{ab}\lambda_{ba}\lambda_{ac}  e_{132}
- \lambda_{ac}\lambda_{ba}\lambda_{ab} e_{23}
- abac(\lambda_{ab}(e_{13} + e_{132}) + \lambda_{ac}(e_{23} + e_{123})) . \label{ababca}
\end{align}
Putting the above computations together, we finally obtain:
\begin{align}
abacabac &= - ab(\lambda_{ab}(e_{13} + e_{132}) + \lambda_{ac}(e_{23} + e_{123}))(\lambda_{bc}(e_{12} + e_{132}) + \lambda_{ba}(e_{13} + e_{123}))ac - ababcbac \notag\\
&\stackrel{(\ref{ababcbac})}{=} - ab(\lambda_{ab}(e_{13} + e_{132}) + \lambda_{ac}(e_{23} + e_{123}))(\lambda_{bc}(e_{12} + e_{132}) + \lambda_{ba}(e_{13} + e_{123}))ac \notag\\
&\phantom{xxx} + ababca(\lambda_{ca}(e_{23} + e_{132}) + \lambda_{cb}(e_{12} + e_{123})) \notag\\
&\phantom{xxx} + abac (\lambda_{bc}(e_{123} + e_{23}) + \lambda_{ba}(e_{1} + e_{12})) (\lambda_{ca}(e_{31} + e_{123}) + \lambda_{cb}(e_{23} + e_{1})) \notag\\
&\stackrel{(\ref{ababca})}{=} - ab(\lambda_{ab}(e_{13} + e_{132}) + \lambda_{ac}(e_{23} + e_{123}))(\lambda_{bc}(e_{12} + e_{132}) + \lambda_{ba}(e_{13} + e_{123}))ac \notag\\
&\phantom{xxx} + (- \lambda_{ab}\lambda_{ba}\lambda_{ac}  e_{132} - \lambda_{ac}\lambda_{ba}\lambda_{ab} e_{23} - abac(\lambda_{ab}(e_{13} + e_{132}) + \lambda_{ac}(e_{23} + e_{123}))) \notag\\
&\phantom{xxxxx}\cdot (\lambda_{ca}(e_{23} + e_{132}) + \lambda_{cb}(e_{12} + e_{123})) \notag\\
&\phantom{xxx} + abac (\lambda_{bc}(e_{123} + e_{23}) + \lambda_{ba}(e_{1} + e_{12})) (\lambda_{ca}(e_{31} + e_{123}) + \lambda_{cb}(e_{23} + e_{1})) \notag\\
&= - ab(\lambda_{ab}(e_{13} + e_{132}) + \lambda_{ac}(e_{23} + e_{123}))(\lambda_{bc}(e_{12} + e_{132}) + \lambda_{ba}(e_{13} + e_{123}))ac \notag\\
&\phantom{xxx} - \lambda_{ab}\lambda_{ba}\lambda_{ac}\lambda_{ca}  e_{132} - \lambda_{ac}\lambda_{ba}\lambda_{ab}\lambda_{ca} e_{23} - abac(\lambda_{ab}\lambda_{ca}e_{132} + \lambda_{ac}\lambda_{ca}e_{23} + \lambda_{ac}\lambda_{cb} e_{123}) \notag\\
&\phantom{xxx} + abac (\lambda_{bc}(e_{123} + e_{23}) + \lambda_{ba}(e_{1} + e_{12})) (\lambda_{ca}(e_{31} + e_{123}) + \lambda_{cb}(e_{23} + e_{1})) \notag\\
&= - abac(\lambda_{ab}(e_{(31)(12)(13)} + e_{(31)(12)(132)}) + \lambda_{ac}(e_{(31)(12)(23)} + e_{(31)(12)(123)})) \notag\\
&\phantom{xxxxx}\cdot (\lambda_{bc}(e_{(31)(12)(12)} + e_{(31)(12)(132)}) + \lambda_{ba}(e_{(31)(12)(13)} + e_{(31)(12)(123)})) \notag\\
&\phantom{xxx} - \lambda_{ab}\lambda_{ba}\lambda_{ac}\lambda_{ca}  e_{132} - \lambda_{ac}\lambda_{ba}\lambda_{ab}\lambda_{ca} e_{23} - abac(\lambda_{ab}\lambda_{ca}e_{132} + \lambda_{ac}\lambda_{ca}e_{23} + \lambda_{ac}\lambda_{cb} e_{123}) \notag\\
&\phantom{xxx} + abac (\lambda_{bc}(e_{123} + e_{23}) + \lambda_{ba}(e_{1} + e_{12})) (\lambda_{ca}(e_{31} + e_{123}) + \lambda_{cb}(e_{23} + e_{1})) \notag\\
&= - abac(\lambda_{ab}(e_{23} + e_{1}) + \lambda_{ac}(e_{12} + e_{132}))(\lambda_{bc}(e_{31} + e_{1}) + \lambda_{ba}(e_{23} + e_{132})) \notag\\
&\phantom{xxx} - \lambda_{ab}\lambda_{ba}\lambda_{ac}\lambda_{ca}  e_{132} - \lambda_{ac}\lambda_{ba}\lambda_{ab}\lambda_{ca} e_{23} - abac(\lambda_{ab}\lambda_{ca}e_{132} + \lambda_{ac}\lambda_{ca}e_{23} + \lambda_{ac}\lambda_{cb} e_{123}) \notag\\
&\phantom{xxx} + abac (\lambda_{bc}(e_{123} + e_{23}) + \lambda_{ba}(e_{1} + e_{12})) (\lambda_{ca}(e_{31} + e_{123}) + \lambda_{cb}(e_{23} + e_{1})) \notag\\
&= - abac(\lambda_{ab}\lambda_{ba}e_{23} + \lambda_{ab}\lambda_{bc}e_{1} + \lambda_{ac}\lambda_{ba}e_{132}) \notag\\
&\phantom{xxx} - \lambda_{ab}\lambda_{ba}\lambda_{ac}\lambda_{ca}  e_{132} - \lambda_{ac}\lambda_{ba}\lambda_{ab}\lambda_{ca} e_{23} - abac(\lambda_{ab}\lambda_{ca}e_{132} + \lambda_{ac}\lambda_{ca}e_{23} + \lambda_{ac}\lambda_{cb} e_{123}) \notag\\
&\phantom{xxx} + abac (\lambda_{bc}\lambda_{ca}e_{123} + \lambda_{bc}\lambda_{cb}e_{23} + \lambda_{ba}\lambda_{cb}e_{1}) \notag\\
&= abac(- \lambda_{ab}\lambda_{ba}e_{23} - \lambda_{ab}\lambda_{bc}e_{1} - \lambda_{ac}\lambda_{ba}e_{132} - \lambda_{ab}\lambda_{ca}e_{132} - \lambda_{ac}\lambda_{ca}e_{23} - \lambda_{ac}\lambda_{cb} e_{123} \notag\\
&\phantom{xxxxxxx} + \lambda_{bc}\lambda_{ca}e_{123} + \lambda_{bc}\lambda_{cb}e_{23} + \lambda_{ba}\lambda_{cb}e_{1}) \notag\\
&\phantom{xxx} - \lambda_{ab}\lambda_{ba}\lambda_{ac}\lambda_{ca}  e_{132} - \lambda_{ac}\lambda_{ba}\lambda_{ab}\lambda_{ca} e_{23} \notag\\
&= abac((- \lambda_{ab}\lambda_{ba} - \lambda_{ac}\lambda_{ca} + \lambda_{bc}\lambda_{cb})e_{23} + (- \lambda_{ab}\lambda_{bc} + \lambda_{ba}\lambda_{cb})e_{1} + (- \lambda_{ac}\lambda_{ba} - \lambda_{ab}\lambda_{ca})e_{132} \notag\\
&\phantom{xxxxxxx}  + ( - \lambda_{ac}\lambda_{cb} + \lambda_{bc}\lambda_{ca})e_{123} ) \notag\\
&\phantom{xxx} - \lambda_{ab}\lambda_{ba}\lambda_{ac}\lambda_{ca}  e_{132} - \lambda_{ac}\lambda_{ba}\lambda_{ab}\lambda_{ca} e_{23} \notag\\
&= abac((\lambda_{ab}^2 + \lambda_{ca}^2 - \lambda_{bc}^2)e_{23} + (- \lambda_{ab}\lambda_{bc} + \lambda_{ab}\lambda_{bc})e_{1} + (2\lambda_{ac}\lambda_{ab})e_{132} \notag\\
&\phantom{xxxxxxx}  + ( - \lambda_{ca}\lambda_{bc} + \lambda_{bc}\lambda_{ca})e_{123} ) \notag\\
&\phantom{xxx} - \lambda_{ab}^2\lambda_{ac}^2  e_{132}
- \lambda_{ac}\lambda_{ab}\lambda_{ab}\lambda_{ac} e_{23} \notag\\
&= abac((\lambda_{ab}^2 + \lambda_{ca}^2 - \lambda_{bc}^2)e_{23} + 2\lambda_{ab}\lambda_{ac}e_{132}) - \lambda_{ab}^2\lambda_{ac}^2 e_{23} - \lambda_{ab}^2\lambda_{ac}^2  e_{132} . \label{abacabac}
\end{align}
\end{proof}

Combining \eqref{abacegabaceg} and Lemma \ref{lem:abacabac} we immediately obtain
\[ (abac\ e_{12})^2 = 0 \text{ and } (abac\ e_{31})^2 = 0 . \]
This implies that
\begin{equation} \label{tr12-31}
\chi_{H^*}(abac\ e_{12}) = 0 \text{ and } \chi_{H^*}(abac\ e_{31}) = 0 .
\end{equation}

Furthermore, using the relations of the algebra $H^*$ -- in particular, we will use several times (we will indicate it when we do) that
\begin{align}
baca  &= - b(cb + ba)a	\notag\\
&= - bcba - b^2a^2	\notag\\
&= bc(ac + cb) - b^2a^2	\notag\\
&= bc(ac + cb) - b^2a^2	\notag\\
&= bcac + bc^2b - b^2a^2	\notag\\
&= - (ab + ca)ac + bc^2b - b^2a^2	\notag\\
&= - abac - ca^2c + bc^2b - b^2a^2	\label{eq:baca},
\end{align}
and using the cyclicity of the trace, we can compute:
\begin{align}
\chi_{H^*}(abac\ e_{23}) &= \chi_{H^*}(baca\ e_{(12)(23)}) \notag\\
&= \chi_{H^*}(acab\ e_{(23)(12)(23)}) \notag\\
&= \chi_{H^*}(acab\ e_{31}) \notag\\
&= \chi_{H^*}(-a(ab + bc)b\ e_{31}) \notag\\
&= \chi_{H^*}(-a^2b^2\ e_{31} - abcb\ e_{31}) \notag\\
&= \chi_{H^*}(-a^2b^2\ e_{31} + ab(ac + ba)\ e_{31}) \notag\\
&= \chi_{H^*}(-a^2b^2\ e_{31} + ab^2a\ e_{31} + abac\ e_{31}) \notag\\
&\stackrel{\eqref{tr12-31}}{=} \chi_{H^*}(-a^2b^2\ e_{31} + a(\lambda_{bc}(e_{12} + e_{132}) + \lambda_{ba}(e_{13} + e_{123}))a\ e_{31}) \notag\\
&= \chi_{H^*}(-a^2b^2\ e_{31}  \notag\\
&\phantom{xxxxxxx} + a^2(\lambda_{bc}(e_{(12)(12)} + e_{(12)(132)}) + \lambda_{ba}(e_{(12)(13)} + e_{(12)(123)})) e_{31} ) \notag\\
&= \chi_{H^*}(-a^2b^2\ e_{31} + a^2(\lambda_{bc}(e_{1} + e_{31}) + \lambda_{ba}(e_{132} + e_{23})) e_{31} ) \notag\\
&= \chi_{H^*}(-a^2b^2\ e_{31} + a^2\lambda_{bc}e_{31} ) \notag\\
&= \chi_{H^*}(-(\lambda_{ab}(e_{13} + e_{132}) + \lambda_{ac}(e_{23} + e_{123})) (\lambda_{bc}(e_{12} + e_{132}) + \lambda_{ba}(e_{13} + e_{123})) e_{31} \notag\\
&\phantom{xxxxxxx} + (\lambda_{ab}(e_{13} + e_{132}) + \lambda_{ac}(e_{23} + e_{123})) \lambda_{bc}e_{31} ) \notag\\
&= \chi_{H^*}(-\lambda_{ab}\lambda_{ba} e_{31} + \lambda_{ab} \lambda_{bc}e_{31} ) \notag\\
&= \chi_{H^*}(\lambda_{ab}(\lambda_{ab} + \lambda_{bc}) e_{31} ) \notag\\
&= \chi_{H^*}(\lambda_{ab}\lambda_{ac} e_{31} ) \notag\\
&= 12 \lambda_{ab}\lambda_{ac}. \label{tr23}
\end{align}
\begin{align}
\chi_{H^*}(abac\ e_{1}) &= \chi_{H^*}(baca\ e_{12}) \notag\\
&\stackrel{\eqref{eq:baca}}{=} \chi_{H^*}( - abac\ e_{12} - ca^2c\ e_{12} + bc^2b\ e_{12} - b^2a^2\ e_{12}) \notag\\
&\stackrel{\eqref{tr12-31}}{=} \chi_{H^*}( - ca^2c\ e_{12} + bc^2b\ e_{12} - b^2a^2\ e_{12}) \notag\\
&= \chi_{H^*}( - c(\lambda_{ab}(e_{13} + e_{132}) + \lambda_{ac}(e_{23} + e_{123}))c\ e_{12} \notag\\
&\phantom{xxxxxxx} + b(\lambda_{ca}(e_{23} + e_{132}) + \lambda_{cb}(e_{12} + e_{123}))b\ e_{12} \notag\\
&\phantom{xxxxxxx} - (\lambda_{bc}(e_{12} + e_{132}) + \lambda_{ba}(e_{13} + e_{123})) \notag\\
&\phantom{xxxxxxxxx} \cdot (\lambda_{ab}(e_{13} + e_{132}) + \lambda_{ac}(e_{23} + e_{123}))\ e_{12} ) \notag\\
&= \chi_{H^*}( - c(\lambda_{ab}(e_{13} + e_{132}) + \lambda_{ac}(e_{23} + e_{123}))c\ e_{12} \notag\\
&\phantom{xxxxxxx} + b(\lambda_{ca}(e_{23} + e_{132}) + \lambda_{cb}(e_{12} + e_{123}))b\ e_{12} ) \notag\\
&= \chi_{H^*}( - c^2(\lambda_{ab}(e_{(13)(13)} + e_{(13)(132)}) + \lambda_{ac}(e_{(13)(23)} + e_{(13)(123)})) e_{12} \notag\\
&\phantom{xxxxxxx} + b(\lambda_{ca}(e_{23} + e_{132}) + \lambda_{cb}(e_{12} + e_{123}))b\ e_{12} ) \notag\\
&= \chi_{H^*}( - c^2(\lambda_{ab}(e_{1} + e_{23}) + \lambda_{ac}(e_{132} + e_{12})) e_{12} \notag\\
&\phantom{xxxxxxx} + b(\lambda_{ca}(e_{23} + e_{132}) + \lambda_{cb}(e_{12} + e_{123}))b\ e_{12} ) \notag\\
&= \chi_{H^*}( - c^2 \lambda_{ac} e_{12} \notag\\
&\phantom{xxxxxxx} + b(\lambda_{ca}(e_{23} + e_{132}) + \lambda_{cb}(e_{12} + e_{123}))b\ e_{12} ) \notag\\
&= \chi_{H^*}( - (\lambda_{ca}(e_{23} + e_{132}) + \lambda_{cb}(e_{12} + e_{123})) \lambda_{ac} e_{12} \notag\\
&\phantom{xxxxxxx} + b^2 (\lambda_{ca}(e_{(23)(23)} + e_{(23)(132)}) + \lambda_{cb}(e_{(23)(12)} + e_{(23)(123)}))  e_{12} \notag\\
&= \chi_{H^*}( - \lambda_{cb} \lambda_{ac} e_{12} + b^2 (\lambda_{ca}(e_{1} + e_{12}) + \lambda_{cb}(e_{132} + e_{31}))  e_{12} \notag\\
&= \chi_{H^*}( - \lambda_{cb} \lambda_{ac} e_{12} + b^2 \lambda_{ca} e_{12} ) \notag\\
&= \chi_{H^*}( - \lambda_{cb} \lambda_{ac} e_{12} + (\lambda_{bc}(e_{12} + e_{132}) + \lambda_{ba}(e_{13} + e_{123})) \lambda_{ca} e_{12} ) \notag\\
&= \chi_{H^*}( - \lambda_{cb} \lambda_{ac} e_{12} + \lambda_{bc} \lambda_{ca} e_{12} ) \notag\\
&= \chi_{H^*}( - \lambda_{cb} \lambda_{ac} e_{12} + \lambda_{bc} \lambda_{ca} e_{12} ) \notag\\
&= 0. \notag
\end{align}
\begin{align}
\chi_{H^*}(abac\ e_{123}) &= \chi_{H^*}(baca\ e_{(12)(123)}) \notag\\
&= \chi_{H^*}(baca\ e_{23}) \notag\\
&\stackrel{\eqref{eq:baca}}{=} \chi_{H^*}( - abac\ e_{23} - ca^2c\ e_{23} + bc^2b\ e_{23} - b^2a^2 e_{23}) \notag\\
&\stackrel{\eqref{tr23}}{=} \chi_{H^*}(- ca^2c\ e_{23} + bc^2b\ e_{23} - b^2a^2 e_{23}) - 12 \lambda_{ab}\lambda_{ac} \notag\\
&= \chi_{H^*}(- c(\lambda_{ab}(e_{13} + e_{132}) + \lambda_{ac}(e_{23} + e_{123}))c\ e_{23} \notag\\
&\phantom{xxxxxxx} + b(\lambda_{ca}(e_{23} + e_{132}) + \lambda_{cb}(e_{12} + e_{123}))b\ e_{23} - b^2a^2 e_{23}) \notag\\
&\phantom{xxx} - 12 \lambda_{ab}\lambda_{ac} \notag\\
&= \chi_{H^*}(- c^2(\lambda_{ab}(e_{(13)(13)} + e_{(13)(132)}) + \lambda_{ac}(e_{(13)(23)} + e_{(13)(123)})) e_{23} \notag\\
&\phantom{xxxxxxx} + b^2(\lambda_{ca}(e_{(23)(23)} + e_{(23)(132)}) + \lambda_{cb}(e_{(23)(12)} + e_{(23)(123)})) e_{23} - b^2a^2 e_{23}) \notag\\
&\phantom{xxx} - 12 \lambda_{ab}\lambda_{ac} \notag\\
&= \chi_{H^*}(- c^2(\lambda_{ab}(e_{1} + e_{23}) + \lambda_{ac}(e_{132} + e_{12})) e_{23} \notag\\
&\phantom{xxxxxxx} + b^2(\lambda_{ca}(e_{1} + e_{12}) + \lambda_{cb}(e_{132} + e_{31})) e_{23} - b^2a^2 e_{23}) \notag\\
&\phantom{xxx} - 12 \lambda_{ab}\lambda_{ac} \notag\\
&= \chi_{H^*}(- (\lambda_{ca}(e_{23} + e_{132}) + \lambda_{cb}(e_{12} + e_{123}))(\lambda_{ab}(e_{1} + e_{23}) + \lambda_{ac}(e_{132} + e_{12})) e_{23} \notag\\
&\phantom{xxxxxxx} + (\lambda_{bc}(e_{12} + e_{132}) + \lambda_{ba}(e_{13} + e_{123}))(\lambda_{ca}(e_{1} + e_{12}) + \lambda_{cb}(e_{132} + e_{31})) e_{23} \notag\\
&\phantom{xxxxxxx} - b^2a^2 e_{23}) \notag\\
&\phantom{xxx} - 12 \lambda_{ab}\lambda_{ac} \notag\\
&= \chi_{H^*}(- \lambda_{ca}\lambda_{ab} e_{23} - b^2a^2 e_{23}) - 12 \lambda_{ab}\lambda_{ac} \notag\\
&= \chi_{H^*}(- \lambda_{ca}\lambda_{ab} e_{23} \notag\\
&\phantom{xxxxxxx} - (\lambda_{bc}(e_{12} + e_{132}) + \lambda_{ba}(e_{13} + e_{123}))(\lambda_{ab}(e_{13} + e_{132}) + \lambda_{ac}(e_{23} + e_{123})) e_{23}) \notag\\
&\phantom{xxx} - 12 \lambda_{ab}\lambda_{ac} \notag\\
&= \chi_{H^*}(- \lambda_{ca}\lambda_{ab} e_{23}) - 12 \lambda_{ab}\lambda_{ac} \notag\\
&= \chi_{H^*}(\lambda_{ac}\lambda_{ab} e_{23}) - 12 \lambda_{ab}\lambda_{ac} \notag\\
&= 12 \lambda_{ac}\lambda_{ab} - 12 \lambda_{ab}\lambda_{ac} \notag\\
&= 0. \notag
\end{align}
\begin{align*}
\chi_{H^*}(abac\ e_{132}) &= \chi_{H^*}(baca\ e_{(12)(132)}) \\
&= \chi_{H^*}(baca\ e_{31}) \\
&\stackrel{\eqref{eq:baca}}{=} \chi_{H^*}( - abac\ e_{31} - ca^2c\ e_{31} + bc^2b\ e_{31} - b^2a^2 e_{31}) \\
&\stackrel{\eqref{tr12-31}}{=} \chi_{H^*}( - ca^2c\ e_{31} + bc^2b\ e_{31} - b^2a^2 e_{31}) \\
&\stackrel{c^2 e_{31} = 0}{=} \chi_{H^*}(bc^2b\ e_{31} - b^2a^2 e_{31}) \\
&= \chi_{H^*}(b(\lambda_{ca}(e_{23} + e_{132}) + \lambda_{cb}(e_{12} + e_{123}))b\ e_{31} - b^2a^2 e_{31}) \\
&= \chi_{H^*}(b^2(\lambda_{ca}(e_{(23)(23)} + e_{(23)(132)}) + \lambda_{cb}(e_{(23)(12)} + e_{(23)(123)})) e_{31} - b^2a^2 e_{31}) \\
&= \chi_{H^*}(b^2(\lambda_{ca}(e_{1} + e_{12}) + \lambda_{cb}(e_{132} + e_{31})) e_{31} - b^2a^2 e_{31}) \\
&= \chi_{H^*}(b^2\lambda_{cb} e_{31} - b^2a^2 e_{31}) \\
&= \chi_{H^*}((\lambda_{bc}(e_{12} + e_{132}) + \lambda_{ba}(e_{13} + e_{123}))\lambda_{cb} e_{31} \notag\\ &\phantom{xxxxxxx}- (\lambda_{bc}(e_{12} + e_{132}) + \lambda_{ba}(e_{13} + e_{123}))a^2 e_{31}) \\
&= \chi_{H^*}(\lambda_{ba}\lambda_{cb} e_{31} - \lambda_{ba}a^2 e_{31}) \\
&= \chi_{H^*}(\lambda_{ba}\lambda_{cb} e_{31} - \lambda_{ba}(\lambda_{ab}(e_{13} + e_{132}) + \lambda_{ac}(e_{23} + e_{123})) e_{31}) \\
&= \chi_{H^*}(\lambda_{ba}\lambda_{cb} e_{31} - \lambda_{ba}\lambda_{ab} e_{31}) \\
&= \chi_{H^*}(\lambda_{ba} (\lambda_{cb} - \lambda_{ba}) e_{31}) \\
&= \chi_{H^*}(\lambda_{ba} \lambda_{ca} e_{31}) \\
&= \chi_{H^*}(\lambda_{ab} \lambda_{ac} e_{31}) \\
&= 12 \lambda_{ab} \lambda_{ac} .
\end{align*}
Summarizing our above calculations we finally obtain:
\begin{proposition}
\begin{align*}
\ph = \frac{1}{\dim H} \chi_{H^*} &= \frac{1}{72} \bigg( \sum_{g \in S_3} 12 g + 12 \lambda_{ab}\lambda_{ac} (\delta_{abac\ e_{23}} + \delta_{abac\ e_{132}}) \bigg) \\ &= \frac{1}{6} \bigg( \sum_{g \in S_3} g + \lambda_{ab}\lambda_{ac} (\delta_{abac\ e_{23}} + \delta_{abac\ e_{132}}) \bigg).
\end{align*}
\end{proposition}

Since the character of the sign representation is $\chi_\sgn = e_1 - e_{12} - e_{23} - e_{31} + e_{123} + e_{132}$, we have therefore:
\[ \phs{\sgn} = \ph(\chi_\sgn \cdot -) = \frac{1}{6} \bigg( 1 - (12) - (23) - (31) + (123) + (132) + \lambda_{ab}\lambda_{ac} (-\delta_{abac\ e_{23}} + \delta_{abac\ e_{132}}) \bigg) .  \]

Furthermore, we have:
\begin{align*}
\phs{V} = \ph(\chi_V \cdot -) &= \frac{1}{6} \bigg( \sum_{g \in S_3} g + \lambda_{ab}\lambda_{ac} (\delta_{abac\ e_{23}} + \delta_{abac\ e_{132}}) \bigg) ((2 e_1 - e_{123} - e_{132})\cdot -) \\ &= \frac{1}{6} ( 2 - (123) - (132) - \lambda_{ab}\lambda_{ac} \delta_{abac\ e_{132}} ).
\end{align*}

Evidently, the images of these elements $\phs{i}$ under the surjection $H \lto \kk S_3$ to the maximal semisimple quotient algebra $\kk S_3$ are the orthogonal idempotents of $\kk S_3$ for the unique isotypic decomposition of $\kk S_3$.
Hence, by Lemma \ref{lem:isotypic-decomposition-as-a-lift-of-the-center}, it only remains to check that the $\phs{i}$ are orthogonal idempotents in $H$ in order to prove that they provide an isotypic decomposition of $H$.

\subsubsection{Idempotence and orthogonality of $\ph$, $\phs{\sgn}$ and $\phs{V}$}

Concerning the question whether these elements $\ph$, $\phs{\sgn}$ and $\phs{V} \in H$ satisfy our Conjecture \ref{conjecture}, we can apply our general results from Subsections \ref{subsec:idempotence-of-conjectured-idempotents} and \ref{subsec:the-phs-sum-up-to-one}.
We obtain from Theorem \ref{thm:idempotents_for_one-dimensional_simples} that $\ph$ and $\phs{\sgn}$ are idempotents projecting to isotypic components of $H$ of trivial type and sign representation type, respectively.
Moreover, we obtain from Theorem \ref{thm:the-phs-sum-up-to-one} that $\ph + \phs{\sgn} + \phs{V} = 1$, if we can verify that $H$ satisfies the assumptions of that proposition.

Indeed, we can ensure that the subalgebra $\Lambda_0 H^* \Lambda_0 \subseteq H^*$ has only one simple representation up to isomorphism.
The Haar integral of the semisimple sub-Hopf-algebra $H^*_0 = \kk^{S_3} \subseteq H^*$ is given by the idempotent $e_1 \in \kk^{S_3}$.
With this we can compute that $\Lambda_0 H^* \Lambda_0 = e_1 H^* e_1 = \kk\{ e_1, abac e_1 \}$, since $1$ and $abac \in B(V)$ span the subspace of $S_3$-degree $1$ in $B(V)$.
$e_1$ is the unit of the algebra $\Lambda_0 H^* \Lambda_0$ and, furthermore, we have $(abac e_1)^2 = 0$ by Lemma \ref{lem:abacabac}.
Essentially, the reason for this is that the deformed relations for the squares $a^2$, $b^2$ and $c^2$ take values in the kernel of right (or left) multiplication by $\Lambda_0 = e_1$.
The algebra $\Lambda_0 H^* \Lambda_0$ is therefore isomorphic to the two-dimensional algebra $\kk[x] / (x^2)$, which indeed has a unique simple representation.

Finally, with the help of a computation with the computer algebra system Magma, as can be seen in the Appendix, we can extend these results to the statement that all three $\ph$, $\phs{\sgn}$ and $\phs{V}$ are idempotents and pairwise orthogonal.

\appendix

\section{Calculations with Magma}

We describe here a Magma code for calculating explicitly the products of the different conjectured idempotents $\phs{i}$  for the Hopf algebra discussed in Subsection \ref{subsec:example-72-dim}.
We begin with the remark that, when considering the grading we have for $H^*$, $$H^*=\bigoplus_{k=0}^4 H^*_k,$$
the only direct summand on which the product $\phs{i} \phs{j}$ (where  $i\neq j$) might not vanish  is $H^*_4$. 
This follows from the fact that all $\phs{i}$'s vanish on $H_k$ for $k\neq 0,4$, and on the coalgebra grading. 
The calculation with Magma will be done in the following way:
for different values of $\lambda_a,\lambda_b,\lambda_c$ we will define the algebra $A=H^*$ in Magma. 
Then we will present it in a matrix form, and calculate the trace of the regular representation, as well as the translations of this trace by multiples of irreducible characters. For the calculations of the product we will calculate $(\phs{i}\ot 1)\Delta(abace_g)$ where $g\in G$ and $\phs{i} \in \{p,p_V\}$ by hand, and apply the relevant functionals $\phs{j}\in \{p,p_{\sgn},p_V\}$ to them. 
Finally, since all the relevant values are polynomials of degree at most 3 in $\lambda_{ab}\lambda_{ac}$ it will be enough to show that they vanish on four different values of $\lambda_{ab}\lambda_{ac}$. 

The code is enclosed here. We ran it on \verb+http://magma.maths.usyd.edu.au/calc/+, the online version of Magma.

\begin{verbatim}
/* Values of lambdas */
lam_a:=0;
lam_b:=23;
lam_c:=11;

K:=RationalField();
A<e1,e2,e3,e4,e5,e6,a,b,c>:= FPAlgebra<K, e1,e2,e3,e4,e5,e6,a,b,c|
e1*e1-e1, e2*e1, e3*e1, e4*e1, e5*e1, e6*e1,
e1*e2, e2*e2-e2, e3*e2, e4*e2, e5*e2, e6*e2,
e1*e3, e2*e3, e3*e3-e3, e4*e3, e5*e3, e6*e3,
e1*e4, e2*e4, e3*e4, e4*e4-e4, e5*e4, e6*e4,
e1*e5, e2*e5, e3*e5, e4*e5, e5*e5-e5, e6*e5,
e1*e6, e2*e6, e3*e6, e4*e6, e5*e6, e6*e6-e6,e1+e2+e3+e4+e5+e6-1,
a*e1-e2*a,a*e2-e1*a, a*e3-e5*a,e5*a-a*e3,a*e4-e6*a,a*e6-e4*a,
b*e1-e3*b,b*e3-e1*b, b*e4-e5*b,e5*b-b*e4,b*e2-e6*b,b*e6-e2*b,
c*e1-e4*c,c*e4-e1*c, c*e2-e5*c,e5*c-c*e2,c*e3-e6*c,c*e6-e3*c,
a*b+b*c+c*a, a*c+c*b+b*a,
a^2 - (lam_a-lam_b)*(e4+e6) - (lam_a-lam_c)*(e3+e5),
b^2 - (lam_b-lam_c)*(e2+e6) - (lam_b-lam_a)*(e4+e5),
c^2 - (lam_c-lam_a)*(e3+e6) - (lam_c-lam_b)*(e2+e5)>;
/* Defining A=H^* by generators and relations */
D:=Dimension(A); 
S,f:= Algebra(A); /* S is now the algebra A considered as a subalgebra of the
72 x 72 matrix algebra. f:A\to S is the natural isomorphism */
Y:=AssociativeArray();

B,h:=ChangeBasis(S,[f(e1),f(e2),f(e3),f(e4),f(e5),f(e6),
f(a*e1),f(a*e2),f(a*e3),f(a*e4),f(a*e5),f(a*e6),
f(b*e1),f(b*e2),f(b*e3),f(b*e4),f(b*e5),f(b*e6),
f(c*e1),f(c*e2),f(c*e3),f(c*e4),f(c*e5),f(c*e6),
f(a*b*e1),f(a*b*e2),f(a*b*e3),f(a*b*e4),f(a*b*e5),f(a*b*e6),
f(b*c*e1),f(b*c*e2),f(b*c*e3),f(b*c*e4),f(b*c*e5),f(b*c*e6),
f(a*c*e1),f(a*c*e2),f(a*c*e3),f(a*c*e4),f(a*c*e5),f(a*c*e6),
f(c*b*e1),f(c*b*e2),f(c*b*e3),f(c*b*e4),f(c*b*e5),f(c*b*e6),
f(a*b*a*e1),f(a*b*a*e2),f(a*b*a*e3),f(a*b*a*e4),f(a*b*a*e5),f(a*b*a*e6),
f(a*b*c*e1),f(a*b*c*e2),f(a*b*c*e3),f(a*b*c*e4),f(a*b*c*e5),f(a*b*c*e6),
f(b*a*c*e1),f(b*a*c*e2),f(b*a*c*e3),f(b*a*c*e4),f(b*a*c*e5),f(b*a*c*e6),
f(a*b*a*c*e1),f(a*b*a*c*e2),f(a*b*a*c*e3),
f(a*b*a*c*e4),f(a*b*a*c*e5),f(a*b*a*c*e6)]);

/* we now fix for S the basis described in the paper.
This is given by the algebra B. The map h:S\to B is then the isomorphism */

for i:=1 to D do
Y[i]:=0;
for j:=1 to D do
Y[i]:= Y[i] + BasisProduct(B,i,j)[j]/72;
end for;
end for;

/* We calculate p as the trace of the regular representation 
divided by the dimension. Notice that we think of p as an element in H=A^*. */

"print p";
for i:=1 to D do
Y[i];
end for;
"end p";
"";

chi:= e1-e2-e3-e4+e5+e6;
chiV:= 2*e1 - e5-e6;

/* the characters of the two non-trivial representations of A^*. 
Both are elements of A */

Z:=AssociativeArray();
for i:=1 to D do
Z[i]:=0;
for j:=1 to D do
Z[i]:= Z[i] + (h(f(chi))*BasisProduct(B,i,j))[j]/72;
end for;
end for;

/* The array Z contains now the translation of p by the sign representation. 
In other words, it is p_{sign}, considered as an element of A^*. */

W:=AssociativeArray();
for i:=1 to D do
W[i]:=0;
for j:=1 to D do
W[i]:= W[i] + 2*(h(f(chiV))*BasisProduct(B,i,j))[j]/72;
end for;
end for;

/* Similarly, we calculate p_V for the 
2-dimensional irreducible representation of A. */ 

E2:=AssociativeArray();
for i:=1 to D do
E2[i]:=Y[i]+Z[i] + 2*W[i];
end for;

"print epsilon";
for i:=1 to D do
E2[i];
end for;
/* We calculate and print the sum p + p_{sign} + p_V.
If it is the counit, then we are on the right path. */ 

"print p_sign";
for i:=1 to D do
Z[i];
end for;
"end p_sign";
"";

"print p_V";
for i:=1 to D do
W[i];
end for;
"end p_V";
"";

/* Next, we calculated manually the 
elements v_i:=(p \otimes 1)\Delta(a*b*a*c*ei). */
v1:= h(f(1/6*(lam_a-lam_b)*(lam_a-lam_c)*(e3 + e5) +
1/6*((lam_c-lam_a)*b*b*e6 - (lam_a-lam_b)*c*c*e4- 
(lam_a-lam_b)*a*a*e5 - (lam_a-lam_c)*a*a*e3)+ 
1/6*(a*b*a*c*e1 + a*c*a*b*e2 + c*b*c*a*e3 + 
b*a*b*c*e4 + c*a*c*b*e6 + b*c*b*a*e5)));

v2:= h(f(1/6*(lam_a-lam_b)*(lam_a-lam_c)*(e6 + e4) +
1/6*((lam_c-lam_a)*b*b*e3 - (lam_a-lam_b)*c*c*e5- 
(lam_a-lam_b)*a*a*e4 - (lam_a-lam_c)*a*a*e6)+
1/6*(a*b*a*c*e2 + a*c*a*b*e1 + c*b*c*a*e6 +
b*a*b*c*e5 + c*a*c*b*e3 + b*c*b*a*e4)));

v3:= h(f(1/6*(lam_a-lam_b)*(lam_a-lam_c)*(e1 + e2) +
1/6*((lam_c-lam_a)*b*b*e4 - (lam_a-lam_b)*c*c*e6-
(lam_a-lam_b)*a*a*e2 - (lam_a-lam_c)*a*a*e1)+
1/6*(a*b*a*c*e3 + a*c*a*b*e5 + c*b*c*a*e1 +
b*a*b*c*e6 + c*a*c*b*e4 + b*c*b*a*e2)));

v4:= h(f(1/6*(lam_a-lam_b)*(lam_a-lam_c)*(e5 + e3) +
1/6*((lam_c-lam_a)*b*b*e2 - (lam_a-lam_b)*c*c*e1- 
(lam_a-lam_b)*a*a*e3 - (lam_a-lam_c)*a*a*e5)+
1/6*(a*b*a*c*e4 + a*c*a*b*e6 + c*b*c*a*e5 + 
b*a*b*c*e1 + c*a*c*b*e2 + b*c*b*a*e3)));

v5:= h(f(1/6*(lam_a-lam_b)*(lam_a-lam_c)*(e4 + e6) +
1/6*((lam_c-lam_a)*b*b*e1 - (lam_a-lam_b)*c*c*e2- 
(lam_a-lam_b)*a*a*e6 - (lam_a-lam_c)*a*a*e4)+
1/6*(a*b*a*c*e5 + a*c*a*b*e3 + c*b*c*a*e4 + 
b*a*b*c*e2 + c*a*c*b*e1 + b*c*b*a*e6)));

v6:=h(f(1/6*(lam_a-lam_b)*(lam_a-lam_c)*(e2 + e1) +
1/6*((lam_c-lam_a)*b*b*e5 - (lam_a-lam_b)*c*c*e3- 
(lam_a-lam_b)*a*a*e1 - (lam_a-lam_c)*a*a*e2)+
1/6*(a*b*a*c*e6 + a*c*a*b*e4 + c*b*c*a*e2 + 
b*a*b*c*e3 + c*a*c*b*e5 + b*c*b*a*e1)));

E:= AssociativeArray();
for i:=1 to 6 do
E[i]:=0;
end for;

for i:= 1 to D do
E[1]:= E[1] + Y[i]*v1[i];
E[2]:= E[2] + Y[i]*v2[i];
E[3]:= E[3] + Y[i]*v3[i];
E[4]:= E[4] + Y[i]*v4[i];
E[5]:= E[5] + Y[i]*v5[i];
E[6]:= E[6] + Y[i]*v6[i];
end for;

/* This calculates p (p \otimes 1)\Delta(a*b*a*c*ei) = p p (a*b*a*c*ei). 
Since a*b*a*c*ei are the only elements on which p^2 might be non-zero, 
it is enough to consider them. 
After that we do a similar calculation for p*p_V and p*p_{sign}.*/ 
"results of p*p - p";
for i:=1 to 6 do
E[i]-Y[D-6+i];
end for;

for i:=1 to 6 do
E[i]:=0;
end for;

for i:= 1 to D do
E[1]:= E[1] + Z[i]*v1[i];
E[2]:= E[2] + Z[i]*v2[i];
E[3]:= E[3] + Z[i]*v3[i];
E[4]:= E[4] + Z[i]*v4[i];
E[5]:= E[5] + Z[i]*v5[i];
E[6]:= E[6] + Z[i]*v6[i];
end for;

"Results of p*p_{sign}";
for i:=1 to 6 do
E[i];
end for;

for i:=1 to 6 do
E[i]:=0;
end for;

for i:= 1 to D do
E[1]:= E[1] + W[i]*v1[i];
E[2]:= E[2] + W[i]*v2[i];
E[3]:= E[3] + W[i]*v3[i];
E[4]:= E[4] + W[i]*v4[i];
E[5]:= E[5] + W[i]*v5[i];
E[6]:= E[6] + W[i]*v6[i];
end for;
"results of p*p_V";
for i:=1 to 6 do
E[i];
end for;

/* Similarly to the elements vi from the previous part, we define 
wi= (p_V \otimes 1)\Delta(a*b*a*c*ei) and similarly calculate the products.*/
w1:= 2*h(f(1/6*((lam_a-lam_c)*b*b*e6 + (lam_a-lam_b)*a*a*e5) -
1/6*((lam_a-lam_b)*(lam_a-lam_c)*e5)
+ 1/3*a*b*a*c*e1 - 1/6*(c*a*c*b*e6 + b*c*b*a*e5)));

w2:= 2*h(f(1/6*((lam_a-lam_c)*b*b*e3 + (lam_a-lam_b)*a*a*e4) -
1/6*((lam_a-lam_b)*(lam_a-lam_c)*e4)
+ 1/3*a*b*a*c*e2 - 1/6*(c*a*c*b*e3 + b*c*b*a*e4)));

w3:= 2*h(f(1/6*((lam_a-lam_c)*b*b*e4 + (lam_a-lam_b)*a*a*e2) -
1/6*((lam_a-lam_b)*(lam_a-lam_c)*e2)
+ 1/3*a*b*a*c*e3 - 1/6*(c*a*c*b*e4 + b*c*b*a*e2)));

w4:= 2*h(f(1/6*((lam_a-lam_c)*b*b*e2 + (lam_a-lam_b)*a*a*e3) -
1/6*((lam_a-lam_b)*(lam_a-lam_c)*e3)
+ 1/3*a*b*a*c*e4 - 1/6*(c*a*c*b*e2 + b*c*b*a*e3)));

w5:= 2*h(f(1/6*((lam_a-lam_c)*b*b*e1 + (lam_a-lam_b)*a*a*e6) -
1/6*((lam_a-lam_b)*(lam_a-lam_c)*e6)
+ 1/3*a*b*a*c*e5 - 1/6*(c*a*c*b*e1 + b*c*b*a*e6)));

w6:= 2*h(f(1/6*((lam_a-lam_c)*b*b*e5 + (lam_a-lam_b)*a*a*e1) -
1/6*((lam_a-lam_b)*(lam_a-lam_c)*e1)
+ 1/3*a*b*a*c*e6 - 1/6*(c*a*c*b*e5 + b*c*b*a*e1)));

/*
"print ws";
for i:= 1 to D do
w1[i], w2[i], w3[i], w4[i], w5[i], w6[i];
end for;
*/
for i:=1 to 6 do
E[i]:=0;
end for;

for i:= 1 to D do
E[1]:= E[1] + Y[i]*w1[i];
E[2]:= E[2] + Y[i]*w2[i];
E[3]:= E[3] + Y[i]*w3[i];
E[4]:= E[4] + Y[i]*w4[i];
E[5]:= E[5] + Y[i]*w5[i];
E[6]:= E[6] + Y[i]*w6[i];
end for;

" ";
"results of p_V*p";
for i:=1 to 6 do
E[i];
end for;

for i:=1 to 6 do
E[i]:=0;
end for;

for i:= 1 to D do
E[1]:= E[1] + Z[i]*w1[i];
E[2]:= E[2] + Z[i]*w2[i];
E[3]:= E[3] + Z[i]*w3[i];
E[4]:= E[4] + Z[i]*w4[i];
E[5]:= E[5] + Z[i]*w5[i];
E[6]:= E[6] + Z[i]*w6[i];
end for;

"Results of p_V*p_{sign}";
for i:=1 to 6 do
E[i];
end for;

for i:=1 to 6 do
E[i]:=0;
end for;

for i:= 1 to D do
E[1]:= E[1] + W[i]*w1[i];
E[2]:= E[2] + W[i]*w2[i];
E[3]:= E[3] + W[i]*w3[i];
E[4]:= E[4] + W[i]*w4[i];
E[5]:= E[5] + W[i]*w5[i];
E[6]:= E[6] + W[i]*w6[i];
end for;

"results of p_V*p_V - p_V";
for i:=1 to 6 do
E[i]-W[D-6+i];
end for; 
\end{verbatim}

\end{document}